\documentclass[12pt,reqno]{amsart}
\usepackage{graphicx}
\usepackage{amsmath}
\usepackage{amssymb}
\usepackage{amscd}
\usepackage{mathrsfs}
\usepackage[all,cmtip]{xy}
\usepackage{amsthm}
\usepackage{caption}

\theoremstyle{definition}
\newtheorem{definition}{Definition}
\theoremstyle{definition}

\theoremstyle{definition}
\newtheorem{remark}{Remark}
\theoremstyle{plain}
\newtheorem{theorem}{Theorem}

\theoremstyle{plain}

\renewcommand\H{{\mathcal H}}

\begin{document}

 \title[Solutions of PDEs with first-order quotients]{Solutions of second-order PDEs with first-order quotients}
 \author{Eivind Schneider}
 \date{}
\address{
Faculty of Science, University of Hradec Králové, Rokitanskeho 62, Hradec Králové 50003, Czech Republic. \newline
E-mail address: {\tt eivind.schneider@uhk.cz}. }

 \begin{abstract}
 We describe a way of solving a partial differential equation using the differential invariants of its point symmetries. By first solving its quotient PDE, which is given by the differential syzygies in the algebra of differential invariants, we obtain new differential constraints which are compatible with the PDE under consideration. Adding these constraints to our system makes it overdetermined, and thus easier to solve. We focus on second-order scalar PDEs whose  quotients are first-order scalar PDEs. This situation occurs only when the Lie algebra of symmetries of the second-order PDE is infinite-dimensional. We apply this idea to several different PDEs, one of which is the Hunter-Saxton equation.
 \end{abstract}
 
  \maketitle
\section{Introduction}
One general approach to finding solutions to a PDE is to add to it additional differential constraints, with the idea that the resulting overdetermined PDE is easier to solve than the original PDE. 
In particular, a sufficiently large number of additional constraints may lead to a PDE whose Cartan distribution is completely integrable. One obstacle for applying this idea in practice comes from the fact that finding compatible differential constraints is a nontrivial task. Most additional differential constraints are incompatible with the original PDE, resulting in an overdetermined PDE with no solutions. 

Another of the general approaches to studying PDEs is to use their symmetries. 
We will consider these two ideas together and show that looking for additional constraints only among differential invariants (of a Lie algebra of point symmetries) makes the problem of finding compatible constraints tractable. One of the essential insights which motivate this approach is that the set of compatibility conditions for such additional constraints is closely related to the quotient of the PDE (the differential syzygies in the algebra of differential invariants). 

The idea is the following: The Lie algebra of symmetries gives an equivalence relation on the space of (local) solutions.  By solving the quotient PDE, we get additional differential constraints that can be added to the system. Roughly speaking, each solution of the quotient PDE corresponds to an equivalence class of solutions to the original PDE, and adding the resulting additional differential constraints to the original PDE then amounts to restricting to a specific equivalence class of solutions. 


Finding the quotient PDE is not trivial in general, but the difficulties lie within the two main fields of algebraic geometry and linear PDEs, which are better understood than the field of nonlinear PDEs. Moreover, for all the PDEs we consider in this paper, finding the quotient is an easy task. Therefore we work under the assumption that the main challenges are to solve the quotient PDE to obtain the additional differential constraints, and to solve the resulting overdetermined PDE.

As was pointed out in \cite{Evolutionary}, we can not, in general, expect the quotient PDE to be easier to solve than the original PDE. 
Therefore we will focus on PDEs with a first-order scalar quotient PDE since these can, at least in principle, be solved by the method of characteristics. We will notice that such quotients appear only when the symmetry Lie algebra is infinite-dimensional. With several examples, some well-known and others specifically constructed, we show how the ideas outlined above can be used to find general solutions to some second-order scalar PDEs on functions of two variables.

Section \ref{Theory} gives an overview of the theory of jet spaces and differential invariants, sufficient for our purpose. For readers new to this subject, we illustrate all concepts using Burgers' equation as a running example. We also provide some Maple code since the computations we are doing here are very well-suited for computer algebra systems. In Section \ref{FirstOrder}, we show that a necessary condition for the quotient to be a first-order scalar PDE only if the Lie algebra of symmetries is infinite-dimensional.

A reader with some knowledge in the theory of differential invariants, or one who is mostly interested in seeing how the ideas work in practice, may wish to jump directly to Section \ref{HS}, where we describe in detail how to find the general solution to the Hunter-Saxton equation. It illustrates the ideas in detail and gives a new perspective on the general solution found by Hunter and Saxton. 

In Section \ref{Main}, we consider other PDEs with infinite-dimensional Lie algebras of symmetries. Details here are sparse, as the purpose is to illustrate the general method with many examples rather than getting lost in the details of each of them. We end the section by looking at two examples again, from a different viewpoint, and find their quotient with respect to a finite-dimensional symmetry Lie algebra. In this case, the quotient is not a scalar PDE, but a system of two, partially uncoupled, first-order PDEs.

\section{Symmetries of PDEs and differential invariants} \label{Theory}
We will focus our attention on differential equations of the form
\begin{equation}
F(t,x,u,u_t,u_x,u_{tt},u_{tx},u_{xx})=0 \label{F}
\end{equation} on a function $u(t,x)$, with a nontrivial Lie algebra $\mathfrak g$ of (infinitesimal point) symmetries. We are going to look for additional differential constraints of the form
\[G(t,x,u,u_t,u_x,u_{tt},u_{tx},u_{xx})=0\]  which are compatible with  (\ref{F}) and $\mathfrak g$-invariant.  They will be built up from differential invariants, and we will find them by solving the quotient PDE of (\ref{F}) with respect to $\mathfrak g$. This section is devoted to explain the necessary concepts from the geometric theory of PDEs and differential invariants. For a more detailed treatment of these topics we recommend \cite{KV,O,KL1,KL2}.

\subsection{The PDE as a manifold} \label{Jets}
Fix a point $a \in \mathbb R^2$, and let $J^k_a(\mathbb R^2)$ denote the set of $k$-degree Taylor polynomials of smooth functions on $\mathbb R^2$ centered at the point $a$. Let $J^k(\mathbb R^2)=\cup_{a \in \mathbb R^2} J^k_a(\mathbb R^2)$, so that $J^k_a(\mathbb R^2)$ is a bundle over $\mathbb R^2$. Denote the projection $J^k_a(\mathbb R^2) \to \mathbb R^2$ by $\pi_k$. For any function $f \in C^\infty_{\text{loc}}(\mathbb R^2)$, we define its $k$-jet $[f]_a^k \in J^k_a(\mathbb R^2)$ at $a \in \mathbb R^2$ to be its $k$-degree Taylor polynomial centered at $a$. 

We will use coordinates $t,x,u,u_t,u_x,...,u_{tx^{k-1}},u_{x^k}$ on $J^k(\mathbb R^2)$. If $a \in \mathbb R^2$ is given by $(t_0,x_0)$, then the coordinates of $\theta=[f]_{a}^k$ are 
\begin{gather*}
t(\theta)=t_0, \quad x(\theta)=x_0, \quad u(\theta)= f(t_0,x_0),\\  u_t(\theta)= \frac{\partial f}{\partial t}(t_0,x_0),\quad  u_x(\theta)= \frac{\partial f}{\partial x}(t_0,x_0),\quad  ...,\\ u_{t x^{k-1}}(\theta) = \frac{\partial^k f}{\partial t \partial x^{k-1}}(t_0,x_0), \quad  u_{x^k}(\theta) = \frac{\partial^k f}{\partial x^k}(t_0,x_0). 
\end{gather*}

By varying $a$, we see that any function $f$ gives rise to a section of $J^k(\mathbb R^2) \to \mathbb R^2$ which we denote by $j^k f$. It is given by $j^k f(t,x)=[f]^k_{(t,x)}$. 

There is additional geometric structure on $J^k(\mathbb R^2)$ responsible for filtering out, from the set of all sections of $\pi_k$, those that are of the form $j^k f$. It is called the Cartan distribution, denoted $\mathcal C^k$. At each point $\theta \in J^k(\mathbb R^2)$, it defines a subspace $\mathcal C^k_\theta \subset T_\theta J^k(\mathbb R^2)$. It is the span of tangent planes of sections of the form $j^k f$ with the property $\theta=j^k f(\pi_k(\theta))$. In $J^1(\mathbb R^2)$ the Cartan distribution is the kernel of the one-form $\omega_0=du-u_t dt-u_x dx$, and in $J^2(\mathbb R^2)$ it is the kernel of the three one-forms $\omega_0, du_t-u_{tt} dt-u_{tx}dx, du_x-u_{tx} dt-u_{xx} dx$. In a similar way we may define the Cartan distribution on $J^{k}(\mathbb R^2)$ as the kernel of $du_{t^i x^{j-i}} - u_{t^{i+1} x^{j-i}} dt-u_{t^{i} x^{j-i+1}} dx$, for $i=0,...,j$ and $j=0,...,k-1$. 

By interpreting (\ref{F}) as an equation on $J^2(\mathbb R^2)$, we obtain a submanifold $\mathcal E_2 \subset J^2(\mathbb R^2)$. The significance of this manifold comes from the following fact: If $f$ is a solution to (\ref{F}) defined on $D \subset \mathbb R^2$, then $j^2 f(D)$ is a two-dimensional submanifold of $\mathcal E_2$. Moreover it is an integral manifold of the restriction of the Cartan distribution to $\mathcal E_2$. 

From this viewpoint we get a natural generalization of the concept of solution to (\ref{F}),  
namely a two-dimensional integral manifold of the Cartan distribution. And we will see in some of the examples below that we may get solutions that are not given globally by a function on $\mathbb R^2$. 

A smooth solution to (\ref{F}) is also a solution to the third-order equations $D_t(F)=0,D_x(F)=0$. We define $\mathcal E_3=\{F=0,D_t(F)=0,D_x(F)=0\} \subset J^3(\mathbb R^2)$,  and similarly, by repeated differentiation, $\mathcal E_k \subset J^k(\mathbb R^2)$. This results in a tower of bundles 
\[\mathbb R^2 \leftarrow J^0(\mathbb R^2)=\mathbb R^2\times\mathbb R \leftarrow J^1(\mathbb R^2) \leftarrow \mathcal E_2 \leftarrow \mathcal E_3 \leftarrow \cdots.\] 
We will also use the notation $\mathcal E_0=J^0(\mathbb R^2)$ and $\mathcal E_1=J^1(\mathbb R^2)$ when convenient.

Since $\dim J^k(\mathbb R^2)= 2+\binom{k+2}{2}$ and $\mathcal E_k$ is given by $\binom{k}{2}$ independent differential constraints, we get 
\[ \dim \mathcal E_k= 3+2k\]
implying that the fibers of $\mathcal E_k \to \mathcal E_{k-1}$ are two-dimensional. Naively, considering formal solutions of $F=0$, we may use this count to estimate the size of the solution space of (\ref{F}). Since $\dim \mathcal E_k=\dim J^k(\mathbb R,\mathbb R^2)$, we expect the solution space to be parametrized by two functions of one variable.

\subsubsection*{Burgers' equation}
We will use Burgers' equation as a running example to illustrate the concepts in this section. Burger's equation is defined by $F=u_{xx}-u_t-uu_x=0$. It defines a seven-dimensional submanifold $\mathcal E_2 \subset J^2(\mathbb R^2)$. Its prolongation $\mathcal E_3$ is defined by \[D_{x}(F)= u_{xxx}-u_{tx}-u_x^2-uu_{xx}=0, \;\; D_t(F)=u_{txx}-u_{tt}-uu_{tx}-u_t u_x=0\] in addition to $F=0$, and is a nine-dimensional submanifold in $J^3(\mathbb R^2)$.

\subsection{Point symmetries}
Let $X$ be a vector field on $J^0(\mathbb R^2)$. In coordinates it takes the form
\[X=a(t,x,u) \partial_t+b(t,x,u) \partial_x+c(t,x,u) \partial_u.\] 
There is a unique vector field $X^{(k)}$ on $J^k(\mathbb R^2)$ that projects to $X$ and preserves the Cartan distribution on $J^k(\mathbb R^2)$. We call it the $k$th prolongation of $X$. The flow of $X^{(k)}$ takes (local) integral manifolds of the Cartan distribution on $J^k(\mathbb R^2)$ to integral manifolds. The formula for $X^{(k)}$ can be found in many introductory treatments of this topic. See for example \cite{KV,KL1}. The computations in this paper are mostly done with the \texttt{DifferentialGeometry} and \texttt{JetCalculus} packages in Maple, where the \texttt{Prolongation} procedure computes the prolongation for us. 

\begin{definition}
A vector field \[X=a(t,x,u) \partial_t+b(t,x,u) \partial_x+c(t,x,u) \partial_u\]
is a (point) symmetry of $F=0$ (or $\mathcal E$) if $X^{(2)}$ is tangent to $\mathcal E_2 \subset J^2(\mathbb R^2)$, i.e. 
\begin{equation}
X^{(2)}(F)|_{\mathcal E_2} = 0. \label{eq:Symmetries}
\end{equation}
\end{definition}
It follows that $X^{(k)}$ is tangent to $\mathcal E_k$ for every $k$. The set of symmetries forms a Lie algebra. Since $X^{(k)}$ preserves the Cartan distribution, its flow acts on the space of (local) integral manifolds of the Cartan distribution on $\mathcal E_k$, and thus on the space of solutions of $F=0$. 

Equation (\ref{eq:Symmetries}) is a polynomial in $u_t,u_x,u_{tt},u_{tx},u_{xx}$, and restricting to $\mathcal E_2$ can be done by using $F=0$ to write one of these coordinates in terms of the others. The vanishing of the remaining polynomial is equivalent to the vanishing of each of its coefficients, which are linear differential equations on $a$, $b$ and $c$. This system of PDEs is often highly overdetermined and not difficult to solve.

\subsubsection*{Symmetries of Burgers' equation}
The Lie algebra of point symmetries of Burgers' equation is spanned by 
\[\partial_t, \quad \partial_x, \quad t \partial_x+\partial_u, \quad t^2 \partial_t+tx \partial_x+(x-tu) \partial_u, \quad 2 t \partial_t+x \partial_x-u\partial_u.\]
We show how the symmetries of Burgers' equation can be found with a few lines of Maple code, since this type of computation is very well-suited for computer algebra systems.
\begin{verbatim}
restart: with(DifferentialGeometry): with(JetCalculus):
DGsetup([t,x],[u], E,2):
F := -u[]*u[2]-u[1]+u[2, 2]:
phi:=Transformation(Prolong(
   DifferentialEquationData([F],[u[2,2]]),3)):
X:=a(t,x,u[])*D_t+b(t,x,u[])*D_x+c(t,x,u[])*D_u[]:
sol:=pdsolve({coeffs(
   expand(Pullback(phi,LieDerivative(Prolong(X,2),F))),
   [u[1],u[2],u[1,1],u[1,2]])}):
eval(X,sol);
\end{verbatim}
This code does mostly symbolic manipulations. However, in the next-to-last line, it uses \texttt{pdsolve} to solve a system of PDEs. In this case the \texttt{pdsolve} procedure is able to find all solutions, but in general, care must be taken when using this procedure and one should not trust it blindly.

While the Maple code is spread throughout Section \ref{Theory}, it should be considered as one unit: Later code may depend on previous code. The main reason for this is that we avoid having to write the three first lines every time.

\subsection{Differential invariants} \label{Invariants}
We continue to consider the arbitrary, but fixed differential equation $F=0$ and its corresponding submanifolds $\mathcal E_k \subset J^k (\mathbb R^2)$. Let $\mathfrak g$ be a Lie algebra of symmetries, possibly a Lie subalgebra of the full symmetry Lie algebra. 

\begin{definition}
A differential invariant of order $k$ is a function on $\mathcal E_k$ that is constant on $\mathfrak g$-orbits.
\end{definition}
This implies that a differential invariant $I \in C_{\text{loc}}^\infty(\mathcal E_k)$ satisfies the PDE  
\begin{equation}
X^{(k)}(I) = 0 \label{eq:Invariant}
\end{equation} 
for every $X \in \mathfrak g$. Even though $X^{(k)}$ may be defined everywhere on $J^k(\mathbb R^2)$, the equation (\ref{eq:Invariant}) is an equation on $\mathcal E_k$ only. In all computations below this is the system of linear PDEs we will solve in order to find a generating set of differential invariants. It is sufficient to check (\ref{eq:Invariant}) on basis elements, and even when $\mathfrak g$ is infinite-dimensional the system will consist of finitely many independent equations (for any fixed order $k$). 

We will exclusively consider invariants whose restrictions to fibers of $\mathcal E_k \to J^0(\mathbb R^2)$ are rational functions, as this is, in most cases of interest, sufficient to separate orbits in general position (see \cite{KL2}).  There are some technical requirements for this, concerning algebraicity of $\mathcal E$ and the symmetry pseudogroup under consideration. We direct the interested reader to \cite{KL2}, as we will not go deep into this topic here.

Let $\mathcal A_k$ denote the field of rational differential invariants of order $k$. We have $\mathcal A_i \subset \mathcal A_{i+1}$ for $i >0$. Let $s_k$ denote the transcendence degree of $\mathcal A_k$. Then $s_k$ is equal to the codimension of a $\mathfrak g$-orbit in $\mathcal E_k$ in general position. Define $\H_k=s_k-s_{k-1}$ and $\H_0=s_0$. The function $\H_k$ of $k$ is called the Hilbert function.
Since $\dim \mathcal E_k=3+2k$ we get $\H_k \leq 2$ for $k \geq 0$. (And $\H_0=3$ only if $\mathfrak g$ is trivial.)

We see that the number of independent differential invariants of order $k$ can, and usually will, increase without bound as $k$ increases. In the next section we introduce invariant derivations, which turn the field of differential invariants into a differential field, which can be generated by a finite number of differential invariants.

\subsubsection*{Differential invariants of Burgers' equation}
We compute second-order differential invariants of Burgers' equation with respect to the three-dimensional symmetry Lie algebra $\mathfrak h = \langle \partial_t, \partial_x, t \partial_x+\partial_u\rangle$ using Maple. 
\begin{verbatim}
sym:=[D_t,D_x,t*D_x+D_u[]]:
pdsolve(Pullback(phi,map(
   LieDerivative,map(Prolong,sym,3),
   f(t,x,u[],u[1],u[2],u[1,1],u[1,2]))));
\end{verbatim}
Notice that this is a situation in where Maple's \texttt{pdsolve} can be safely used. Since we have other ways of knowing how many independent invariants exist, \texttt{pdsolve}'s output can easily be checked. 

There are four second-order invariants:
\[u_x, \!\!\quad u_t+uu_x,\!\! \quad u_{tx}+u(u_t + uu_x),\!\! \quad u_{tt}+4 u_t u_x+2(2 u_x^2+u_{tx})+u^2(u_t+uu_x)\] 
Here we have used the variable $t,x,u,u_t,u_x,u_{tt},u_{tx}$ as coordinates on $\mathcal E_2$. Notice that these invariants can also be given by 
\[I = u_x, \qquad J=u_{xx}, \qquad H=u_{xxx}, \qquad K=u_{xxxx}. \]
The rewriting may be done by using $F=0$ and its derivatives. In Maple the rewriting can be done like this:
\begin{verbatim}
A:=u[2]: B:=u[2,2]: H:=u[2,2,2]: K:=u[2,2,2,2]:
Pullback(phi,[A,B,H,K]);
\end{verbatim}
We named the invariants \texttt{A} and \texttt{B} in Maple, instead of $I$ and $J$, because Maple's \texttt{I} is the imaginary unit. 
Since $\dim \mathcal E_2=7$, and $\mathfrak h$ orbits are three-dimensional (the action is free already on $J^0(\mathbb R^2)$), the transcendence degree of the field of second-order differential invariants is 4. In the next section we will show how to generate the rest of the differential invariants.

If $s$ is a solution to $\mathcal E$, given by a function $f$ on $D \subset \mathbb R^2$, the restriction of a $k$th-order invariant $I$ to $s$ is given by $I_s=I\circ j^k f$. 

\subsection{Invariant derivations}
The only invariant derivations we will encounter in this paper are the so-called Tresse derivatives. They are a commuting pair of invariant derivations that play the roles of partial derivatives with respect to a pair of independent differential invariants.  In order to construct them, it will be useful to have the notion of horizontal differential.  The horizontal differential $\hat d$ on a function $f$ on $J^k(\mathbb R^2)$ (or on $\mathcal E_k$) is given in coordinates by 
\[ \hat df= D^{\mathcal E}_t(f) dt+D^{\mathcal E}_x(f) dx\]
where $D^{\mathcal E}_t$ and $ D^{\mathcal E}_x$ are the restrictions of the total derivatives 
\[D_t= \partial_t+u_t \partial_u+ u_{tt} \partial_{u_t}+u_{tx} \partial_{u_x}+\cdots , D_t= \partial_x+u_x \partial_u+u_{tx} \partial_{u_t}+u_{xx} \partial_{u_x}+\cdots\]
to $\mathcal E$. If $s$ is a solution of $F=0$, then $(\hat df)_s= d(f_s)$.  

Let $I,J$ be two differential invariants of order $k$ with the property 
\begin{equation}
\hat dI \wedge \hat dJ\neq 0. \label{dIdJ}
\end{equation} 
In general, this unequality will hold on a Zariski-open set in $\mathcal E_{k+1}$, and all subsequent computations we do are restricted to this Zariski-open set, even if it is not mentioned explicitly. In particular, when we are going to compute solutions of $F=0$, we can not expect to find solutions whose $(k+1)$-jets lie outside this set. 

Assuming that (\ref{dIdJ}) holds, we call the pair $\hat dI, \hat dJ$ a horizontal coframe (for a solution $s$ in general position, the pair $(\hat dI)_s, (\hat dJ)_s$ will give a coframe on the two-dimensional manifold $s$). Its dual frame consists of derivations, which we denote by $\hat \partial_I, \hat \partial_J$, that are of the form $\alpha D_t+\beta D_x$ and satisfy \[\hat dI(\hat \partial_I)=1, \quad \hat dI(\hat \partial_J)=0, \quad \hat dJ(\hat \partial_I)=0,\quad  \hat dJ(\hat \partial_J)=1.\] Here $\alpha,\beta$ are functions on $\mathcal E_{k+1} \subset J^{k+1}(\mathbb R^2)$. 

The derivations $\hat \partial_I$ and $\hat \partial_J$ commute, and they are invariant with respect to $\mathfrak g$, in the sense that 
\[[X^{(\infty)},\hat \partial_I]=0,\qquad  [X^{(\infty)},\hat \partial_J]=0\] for every $X \in \mathfrak g$. In general, if $H$ is an invariant of order $l\geq k$, then $\hat \partial_I(H)$ and $\hat \partial_J(H)$ will be invariants of order $l+1$, and if we apply the Tresse derivatives to $I$ and $J$ we get \[\hat \partial_I(I)=1, \quad \hat \partial_I(J)=0, \quad \hat \partial_J(I)=0, \quad \hat \partial_J(J)=1.\] This explains the interpretation of the Tresse derivatives as partial derivatives with respect to $I$ and $J$, respectively: When restricted to a solution $s$, they become the partial derivatives with respect to $I_s$ and $J_s$. Notice also that both $\hat \partial_I$ and $\hat \partial_J$ depends on the pair $(I,J)$. Thus, if one of the invariants $I$ or $J$ are changed, both derivations will change.

\begin{theorem}
The algebra of differential invariants is generated by a finite number of differential invariants $I,J,H_1,...,H_q$ together with the invariant derivations $\hat \partial_I, \hat \partial_J$. 
\end{theorem}
Again we refer to \cite{KL2} for the general theory.

\subsubsection*{Tresse derivatives for Burgers' equation}
In Maple the Tresse derivatives can be computed like this (remember that we defined the invariants \texttt{A} and \texttt{B} in the previous subsection):
\begin{verbatim}
T:=proc(f) a*TotalDiff(f,t)+b*TotalDiff(f,x); end proc:
cf1:=eval([a,b],solve(map(T,[A,B],a,b)-[1,0],{a,b})):
cf2:=eval([a,b],solve(map(T,[A,B],a,b)-[0,1],{a,b})):
T1:=proc(f) cf1[1]*TotalDiff(f,t)+cf1[2]*TotalDiff(f,x); 
   end proc:
T2:=proc(f) cf2[1]*TotalDiff(f,t)+cf2[2]*TotalDiff(f,x); 
   end proc:
simplify([T1(A),T1(B),T2(A),T2(B)]);
\end{verbatim}
They are given by 
\begin{equation*}
\hat \partial_I = \frac{u_{xxx} D_t-u_{txx} D_x}{u_{tx} u_{xxx}-u_{xx} u_{txx}}, \qquad \hat \partial_J = \frac{-u_{xx} D_t+u_{tx} D_x}{u_{tx} u_{xxx}-u_{xx} u_{txx}}.
\end{equation*}
Note that the coefficients are functions on $\mathcal E_2$, but their expressions are simpler when written like this.

\subsection{The Quotient PDE} \label{Quotient}
Assume that the algebra of differential invariants is generated by the invariants $I,J,H_1,...,H_q$ and the Tresse derivatives $\hat \partial_I, \hat \partial_J$. In general this algebra will not be freely generated: there are differential syzygies, i.e. relations between $I,J,H_i, \hat \partial_I H_i, \hat \partial_J H_i$ and higher order derivatives. The differential syzygies define what we call the quotient PDE. Its meaning can be explained as follows.

Let us restrict the invariants to a solution $s$ of $F=0$, and denote the obtained functions by $I_s,J_s,(H_i)_s$, respectively. These can be viewed (locally) as functions on $\mathbb R^2$. Since we have $2+q$ functions on a two-dimensional manifold, there must be at least $q$ independent relations between them. We will consider only such solutions that $I_s$ and $J_s$ are independent (the $(k+1)$-jets of $s$ satisfy (\ref{dIdJ}), where $k$ is the order of $I$ and $J$). Locally we may solve for $(H_i)_s$, so that \[ (H_i)_s=h_i(I_s,J_s)\] for some set of functions $h_i$ of two variables. At first glance it looks like the functions $h_i$ depend on $s$, but in fact, since $I,J,H_i$ are invariant, $h_i$ only depends on the equivalence class of $s$ (where the equivalence relation is determined by the Lie algebra of symmetries). Moreover, $s$ is a solution of both the PDE $F=0$ and the PDE system $H_1=h_1(I,J),..., H_q=h_q(I,J)$. 

The differential syzygies imply that the functions $h_1,...,h_q$ are not completely arbitrary. Instead, they satisfy a system of differential equations. By differentiating $H_i=h_i(I,J)$ with respect to $\hat \partial_I$ and $\hat \partial_J$ we get 
\[\hat \partial_I(H_i)= (h_i)_1(I,J), \qquad \hat \partial_J(H_i)= (h_i)_2(I,J).\]
where $(h_i)_j$ is denotes the partial derivative of $h_i$ with respect to its $j$th argument. The right-hand sides are just functions of $I$ and $J$, while left-hand sides are new differential invariants that are related by the differential syzygies. In this way the differential syzygies can be identified with differential equations on the functions $h_i$ (it may be necessary to differentiate multiple times in order to get all equations on $h_i$, in case there exist syzygies that are not generated by those of first order). 

One reason these ideas are useful for solving PDEs is the following. Let us add differential constraints $H_i=h_i(I,J)$ in such a way that the equations $F=0,H_1=h_1(I,J),...,H_q=h_q(I,J)$ are compatible. One way to check compatibility is to differentiate the expressions and make sure that there are not new equations of equal or lower order appearing.  If we differentiate with respect to the Tresse derivatives, it is clear that the compatibility conditions are exactly the differential syzygies. 
This connection between the quotient PDE and compatible additional differential constraints is one of the main motivations of this paper.

\subsubsection*{Quotient of Burgers' equation} \label{Burgers}
Let's compute the quotient with respect to the three-dimensional Lie algebra $\mathfrak h$ for which we already found differential invariants. The differential syzygies are found by differentiating $H$ and $K$ with respect to $\hat \partial_I, \hat \partial_J$, and then looking for relations among $I,J,H,K, \hat \partial_I H, \hat \partial_J H, \hat \partial_I K, \hat \partial_J K$ (as functions on $\mathcal E_3$). 
\begin{verbatim}
eliminate(Pullback(phi,[T1(H)-h[1],T2(H)-h[2],T1(K)-k[1],
   T2(K)-k[2],K-k[],H-h[],A-a,B-b]),
   [u[1,1,1],u[1,1,2],u[1,1],u[1,2],u[1],u[2]])[2];
\end{verbatim}
The quotient PDE is given by the two equations
\begin{gather*}
(I^2 H-3 I J^2+J K-H^2) \hat \partial_I(H)+J H \hat \partial_I(K)+H^2 \hat \partial_J(K)-K^2 \\+3 I J K-4 I H^2-3 J^2 H=0,\quad  J  \hat \partial_I(H)+H  \hat \partial_J(H)-K =0. 
\end{gather*}
Now, imagine that we were able to solve this system. Any solution to this system can be given implicitly by $G_1(I,J,H,K)=0$ and $G_2(I,J,H,K)=0$, where $G_1$ and $G_2$ are fixed functions of four variables. In this way the solution gives rise to two additional differential constraints on $J^2(\mathbb R^2)$ that can be added to the original equation $F=0$. This gives a five-dimensional submanifold $\{F=0,G_1=0,G_2=0\} \subset\mathcal E_2 \subset  J^2(\mathbb R^2)$. The restriction of the Cartan distribution to this manifold is two-dimensional and completely integrable, and the symmetry group acts transitively on its leaves. The solvability of the symmetry Lie algebra allows for the application of the Lie-Bianchi theorem, meaning that solutions can be found by quadratures (see for example \cite{RedBook}). 

Notice that the second equation shows that $K$ can be generated by the other invariants, so that $I,J,H$ generate the algebra of invariants. Using only these three generators, the quotient PDE is given by 
\begin{gather*}
J^2  \hat \partial_I^2(H)+2 J H  \hat \partial_I \hat \partial_J(H)+H^2  \hat \partial_J^2(H) \\ +I^2  \hat \partial_I(H)+3 I J  \hat \partial_J(H)-4 I H-3 J^2=0.
\end{gather*}
Quotients of evolutionary PDEs are treated in \cite{Evolutionary}.

\begin{remark}
We chose this three-dimensional Lie algebra of symmetries for our computations here instead of the five-dimensional one because we wanted $I,J,H,K$ to be of second order. If ones goal is to get an understanding of the space of equivalence classes of solutions, one should consider the whole Lie algebra of symmetries, or the part of it that one is interested in. Our  purpose here is to find solutions of PDEs, and for this it would be counter-productive to choose the largest possible Lie algebra of symmetries. For the interested reader we have written down the quotient of the full Lie algebra of symmetries of Burgers' equation in the appendix.
\end{remark}

\begin{remark}
Burgers' equation is the quotient of the heat equation $u_t=u_{xx}$ with respect to the one-dimensional symmetry Lie algebra spanned by $u \partial_u$. 
\end{remark}

\subsection{Solving ODEs by symmetry reduction} \label{ODEs}
A special case of the methods used in this paper appears in the case of ODEs where symmetries allows us to reduce the order (see for example \cite{BlumanODE} and \cite{O}). 
Consider the ODE given by 
\begin{equation}
 u_k=F(x,u,u_1,...,u_{k-1})=0 \label{eq:ode}
 \end{equation}
  where $u_i$ denotes the $i$th derivative of $u(x)$. The equation (\ref{eq:ode}) determines a $(k+1)$-dimensional submanifold $\mathcal E_k$ of $J^k(\mathbb R)$. Let $\mathfrak g$ be a Lie algebra of symmetries whose generic orbits on $\mathcal E_k$ are $r$-dimensional. The number of independent differential invariants on $\mathcal E_k$ is $(k+1-r)$, and they can be generated by two invariants $I$ and $H$ and derivatives of $H$ with respect to the Tresse derivative
\[\hat \partial_I= \frac{1}{D_x(I)} D_x.\] 
Since there are only $(k+1-r)$ independent invariants, the invariant $\hat \partial_I^{k-r}(H)$ can be written in terms of the other. Thus we get a differential syzygy, or quotient equation, which now is an ODE of order $(k-r)$: 
\[ G(I,H,\hat \partial_I (H) ,..., \hat \partial_I^{k-r}(H))=0\]
Each solution $g(I,H)=0$ of this can be considered as an $r$th-order ODE that has (\ref{eq:ode}) as a differential consequence ($I$ and $H$ can be chosen to be invariants of order $r$). Thus we may, instead of solving one ODE of order $k$, solve one ODE of order $(k-r)$ and one of order $r$.

\subsubsection*{Example: The quotient and solution of an ODE}
We chose the following example for its complete transparency, and because it shows a subtle detail of quotients of differential equations that can be useful to keep in mind. Consider the ODE defined by 
\[ u_{xxx}=u_{xx}\] 
and the Lie algebra $\mathfrak g = \langle \partial_x, \partial_u \rangle$ of symmetries. The four-dimensional ODE is foliated by two-dimensional orbits. The field of rational differential invariants is generated by $I=u_x$ and $H=u_{xx}$, and these two invariants completely parametrize all orbits. The Tresse derivative is given by $\hat \partial_I= \frac{1}{u_{xx}} D_x$, so that the quotient is $\hat \partial_I(H)=1$. We assume that $u_{xx} \neq 0$ (the case $u_{xx}=0$ can be treated separately).  

The solution of the quotient is $H=I-A$, which gives us a new ODE: $u_{xx}=u_x-A$. Its derivative is the ODE we started with. The solution of  this second-order ODE is $u(x)=Ax+B+C e^x$. Thus we obtained the general solution to $u_{xxx}=u_{xx}$. 

Notice that even though $I$ and $H$ completely separate $\mathfrak g$-orbits on the ODE, the relation $H=I-A$ alone does not completely separate inequivalent solutions. This can be seen by switching the constant $C$ with $-C$ which changes the equivalence class of the solution while preserving the relation $H=I-A$. These solutions are separated by the hypersurface given by $H=u_{xx}=0$. For every $A$, there is one equivalence class with $H<0$, and one with $H>0$.  

\subsection{First-order quotients} \label{FirstOrder}
The method above shows that solving the PDE $F=0$ can be broken down into two steps: solving the quotient PDE, and solving the overdetermined system made by adding the differential constraints corresponding to a solution of the quotient. Intuitively, we restrict to solutions lying  inside one equivalence class (although the previous example shows that it is not that simple). The hope is that each of these two steps will be significantly easier than solving the original PDE. In the case of Burgers' equation, the last step is easy, but the first is not. 

The goal of this section is to pin down a special case in which we are able to solve the quotient PDE. Since first-order scalar PDEs can be solved by using the method of characteristics, at least in principle (see for example \cite{KV}), we seek PDEs with a first-order scalar quotient.

In all of the examples we are going to consider, we will need to solve two first-order PDEs, one of which is the quotient.  We noted in Section \ref{Jets} that the general solution is expected to be parametrized by two functions of one variable. In all the examples below, these appear as the parameters in the general solution of the two first-order PDEs. Thus, in this case the picture is very similar to that of ODEs.

Assume that the pair $(\mathcal E,\mathfrak g)$ has a first-order scalar quotient PDE. 
Reformulating this in terms of the algebra of differential invariants, we require it to be generated by three independent differential invariants $I,J,K$ and the Tresse derivatives $\hat \partial_I, \hat \partial_J$. The quotient is of desired type if and only if there is one syzygy of the form $S(I,J,K,\hat \partial_I (K),\hat \partial_J (K))=0$ and all other differential syzygies are generated by it and its Tresse derivatives. We also require that the syzygy $S=0$ and its differential consequences $\hat \partial_I(S)=0, \hat \partial_J(S)=0,...$ are independent. 


If we differentiate one more time, we get three new invariants $\hat \partial_I^2(K)$, $\hat \partial_I \hat \partial_J(K)$ and $\hat \partial_J^2(K)$, and two new syzygies: 
\begin{align*}
\hat \partial_I(S)&= S_1+S_3 \hat \partial_I(K)+S_4 \hat \partial_I^2(K)+S_5 \hat \partial_I \hat \partial_J(K)=0,
\\ \hat \partial_J(S)&=S_2+S_3 \hat \partial_J(K)+S_4 \hat \partial_I \hat \partial_J(K)+S_5 \hat \partial_J^2(K)=0.
\end{align*} 
Now we have eight differential invariants, and three syzygies between them. Continuing this we see that we obtain at each step exactly one differential invariant that is independent of all the invariants from the previous step. At some point we must obtain an invariant $L$ that has order $l$ higher than $I$ and $J$, so that at least one of $\hat \partial_I (L)$ or $\hat \partial_J(L)$ is of order $l+1$. Then further differentiation will give invariants of strictly increasing order at each step. 

We saw in Section \ref{Invariants} that, for the Hilbert function $\H_k$, we have $\H_k \leq 2$ for $k \geq 0$. When the differential invariants are generated by three invariants and one differential syzygy, and all derivatives of the syzygy are independent, we see that $\H_k=1$ for every $k$ above some integer. This implies that the dimension of a $\mathfrak g$-orbit in $\mathcal E_k$ grows without bound, as $k$ grows. Thus $\mathfrak g$ is infinite-dimensional. 

\begin{theorem}
Let $\{F=0\}\subset J^2(\mathbb R^2)$ be a second-order PDE, and let $\mathfrak g$ be a Lie algebra of symmetries. Its quotient PDE is a first-order PDE on one function of two variables only if $\mathfrak g$ is infinite-dimensional. 
\end{theorem}

With this in mind, we look for second-order PDEs whose symmetry Lie algebra has infinite dimension.

\section{The Hunter-Saxton equation} \label{HS}
In this section we will use the Hunter-Saxton equation to show in detail how the ideas outlined above can be used to find exact solutions of PDEs. We will find a formula for the general solution, similar to that found in \cite{HS}, and we will look at a few concrete solutions, also with respect to Cauchy data. 

Here, and in the rest of the paper, we will continually recycle the notation used above. The second-order PDE under consideration will be denoted by $\mathcal E$ and given by $F=0$, the symmetry Lie algebra under consideration is $\mathfrak g$, and $I,J,H$ are generators for the algebra of differential invariants, and so on.

The Hunter-Saxton equation is given by \[F=(u_t+uu_x)_x-u_x^2/2=0.\]
It was derived in \cite{HS} in order to describe nonlinear instability in the director field of a nematic liquid crystal. 

The Lie algebra of symmetries of the Hunter-Saxton equation is spanned by
\begin{align*}
X_1&=\partial_t, \quad X_2= t \partial_t+x \partial_x, \quad X_3=x \partial_x+u\partial_u, \\
 X_4&=t^2 \partial_t+2tx \partial_x+2x \partial_u,\quad  Y_f = f(t) \partial_x+f'(t) \partial_u,
\end{align*}
where $f$ runs through all smooth functions on $\mathbb R$ (\cite{HSsymmetries}). We will consider only the infinite-dimensional Lie subalgebra $\mathfrak g=\langle Y_f \mid f \in C_{\text{loc}}^\infty(\mathbb R)\rangle$.

\subsection{Differential invariants and the quotient PDE}
By choosing an infinite-dimensional Lie algebra of symmetries, we get a first-order quotient PDE. At the same time, by not considering a larger symmetry Lie algebra, we only have to look to $\mathcal E_2$ to find a generating set of invariants.
\begin{theorem}
The algebra of rational differential invariants is generated by  \[I=t, \quad J= u_x, \quad H=u_{xx}\] together with the Tresse-derivatives \[ \hat \partial_I = D_t-\frac{u_{tx}}{u_{xx}} D_x, \quad \hat \partial_J=\frac{1}{u_{xx}} D_x.\] 
\end{theorem}
\begin{proof}
It is easy to verify that $I,J,H$ are invariant. What remains to show is that they generate the whole algebra. 

Differentiation of $H$ with respect to Tresse derivatives clearly results in one new independent differential invariant of each order. This shows that $\H_k\geq 1$, which gives an upper bound on dimension of $\mathfrak g$-orbits in $\mathcal E_k$. We show that this upper bound is obtained.

Let $Z_i$ denote the restriction of $Y_{t^{i+1}}^{(i)}$ to the subset in $\mathcal E_i$ given by $t=0,u_x=0$, and $Z_{-1}=\partial_x$. Then  $Z_i=(i+1)! \partial_{u_{t^{i}}}$. At every point under consideration in $\mathcal E_k$, the vectors  $\{Z_i \mid i=-1,...,k\}$ are independent on points in general position in $\mathcal E_k$. It follows that orbits in general position in $\mathcal E_k$ have dimension greater than or equal to $k+2$, confirming that $\H_k=1$ for every $k$. 
\end{proof}
The quotient PDE is given by 
\[2  \hat \partial_I(H)-J^2 \hat \partial_J(H)+4 J H =0.\] 

\begin{remark} The differential invariants $I$ and $J$ are independent for generic solutions. Thus each solution gives a  relation of the form $G=H-h(I,J)=0$.  The condition that $F=0$ and $G=0$ are compatible, puts a restriction on $h$ which is found by  differentiating the system $\{F=0,G=0\}$, and then eliminating $u_{ttx},u_{txx},u_{xxx}$. The result is a linear first-order PDE on the function $h$: 
\[2  h_I-J^2  h_J+4 J h =0.\] 
Equivalent solutions of the Hunter-Saxton equation will result in the same function $h$. 
\end{remark}

Its general solution is easily found to be given implicitly by 
\[16 g\left(\tfrac{2J}{2-IJ}\right) H-(2-IJ)^4=0.\] 
We choose to write it down like this for future convenience. 
By inserting the expressions for $I,J,H$, we get 
\begin{equation}
G=16 g\left(\tfrac{2u_x}{2-t u_x }\right) u_{xx}-(2-tu_x)^4 =0. \label{G}
\end{equation}

\subsection{The general solution} 
For every choice of $g$, the functions $F=0$ and $G=0$ are compatible. The equation $G=0$ can be viewed as a first-order PDE on $u_x$ which we can solve (as a first-order separable ODE). Its solution is given implicitly by
\[ \hat G =  \left(\int (tw+2)^2 g(w) dw\right)\Bigg|_{w=\tfrac{2u_x}{2-t u_x}}+4C(t)-4x =0.\]
The choice of anti-derivative of $(tw+2)^2g(w)$ is obviously not essential here, since $C(t)$ is arbitrary. However, it will be convenient to fix this, so let us use the convention $\int a(w) dw= \int_0^w a(v) dv$. 
The Lie algebra $\mathfrak g$ acts transitively on the set of integration ``constants'' $C(t)$. 

The four equations $F=0, G=0, \hat G=0, D_t \hat G=0$ defines a two-dimensional surface in the space $\mathbb R^6(t,x,u,u_x,u_{tx},u_{xx})$. Its projection to $\mathbb R^3(t,x,u)$ is a solution to the Hunter-Saxton equation. Notice that due to the symmetries, $u_t$ and $u_{tt}$ do not appear in any of these equations. 

In this case we are even able to give the two-dimensional manifold as a parametrized surface in $\mathbb R^3$. We may solve $D_t \hat G=0, G=0$ for $u_{tx},u_{xx}$ and eliminate second-order derivatives from $F=0$ to obtain an equation $\hat F=0$ depending only on $t,x,u,u_x$. 
Since $x$ appears only in $\hat G=0$ and $u$ appears only in $\hat F=0$, both in a linear way, may solve $\hat G=0, \hat F=0$ for $x$ and $u$ to obtain the solution in $\mathbb R^3(t,x,u)$, parametrized by $I=t,J=u_x$, or, in order to simplify the expression, by $t$ and $w=\frac{2u_x}{2-tu_x}$: 
\begin{equation}
\begin{cases}
t &= t \\
x &= \frac{1}{4} \int (tw+2)^2 g(w) dw+ C(t)\\
u&=\frac{1}{2} \int (tw+2)wg(w) dw+C'(t)
\end{cases} \label{eq:HSsol}
\end{equation}
This two-dimensional manifold will in general be multivalued and have points where it's not differentiable, even when $g$ is smooth.  Notice also that the general solution depends on two arbitrary functions of one variable, $g$ and $C$. 

Let $s$ be any of these parametrized solutions, and consider its intersection $s_{t_0}$ with the plane in $\mathbb R^3$ given by $t=t_0$. Since $u_x$ is a parameter, the function $u_x$ restricted to the curve $s_{t_0}$ is injective. In other words: at any point in time, the slope $u_x(t_0,x)$ will be different for every $x$ where it is defined.

  
\subsection{Comparing to the solution of Hunter and Saxton}
Hunter and Saxton solved their equation in \cite{HS}:
\begin{theorem}[Hunter-Saxton]
Every smooth solution of the Hunter-Saxton equation with the Cauchy data $u(0,x)=\alpha(x)$ is given implicitly by 
\begin{align*}
u&=\alpha(\xi)+t \beta(\xi)+\gamma'(t) \\
x&=\xi + t \alpha(\xi)+\frac{1}{2} t^2 \beta(\xi)+\gamma(t)
\end{align*}
where $\gamma$ is any function with $\gamma(0)=\gamma'(0)=0$ and $\beta$ satisfies $\beta'(\xi)=\frac{1}{2} \alpha'(\xi)^2$. 
\end{theorem}
We see the relation between our formula for the general solution and this one by writing $\gamma(t)=C(t)$ and 
\begin{align*}
\xi= \int g(w) dw, \quad \alpha(\xi)= \int g(w) w dw, \quad \beta(\xi)=\frac{1}{2} \int g(w) w^2 dw.
\end{align*}
It follows that $w=\alpha'(\xi)$.

\subsection{Two examples}
Let us look at a few examples for different functions $g$, and $C$. When $g(w)=e^w$ and $C(t) \equiv 0$ the solution is given by
\begin{align*}
x &= \frac{(t(t u_x-2)+2)^2+4}{2(2-tu_x)^2} \exp \left({\frac{2u_x}{2-tu_x}}\right), \\
u &= \frac{t(2-u_x t)^2+t^2 u_x^2+4 u_x-4}{(2-tu_x)^2} \exp \left({\frac{2u_x}{2-tu_x}}\right).
\end{align*}
We notice that the solution is defined for positive $x$ only and it is not smooth everywhere. See Figure \ref{gexp}.

\begin{figure}[h]
\centering
\fbox{\includegraphics[width=3.8cm]{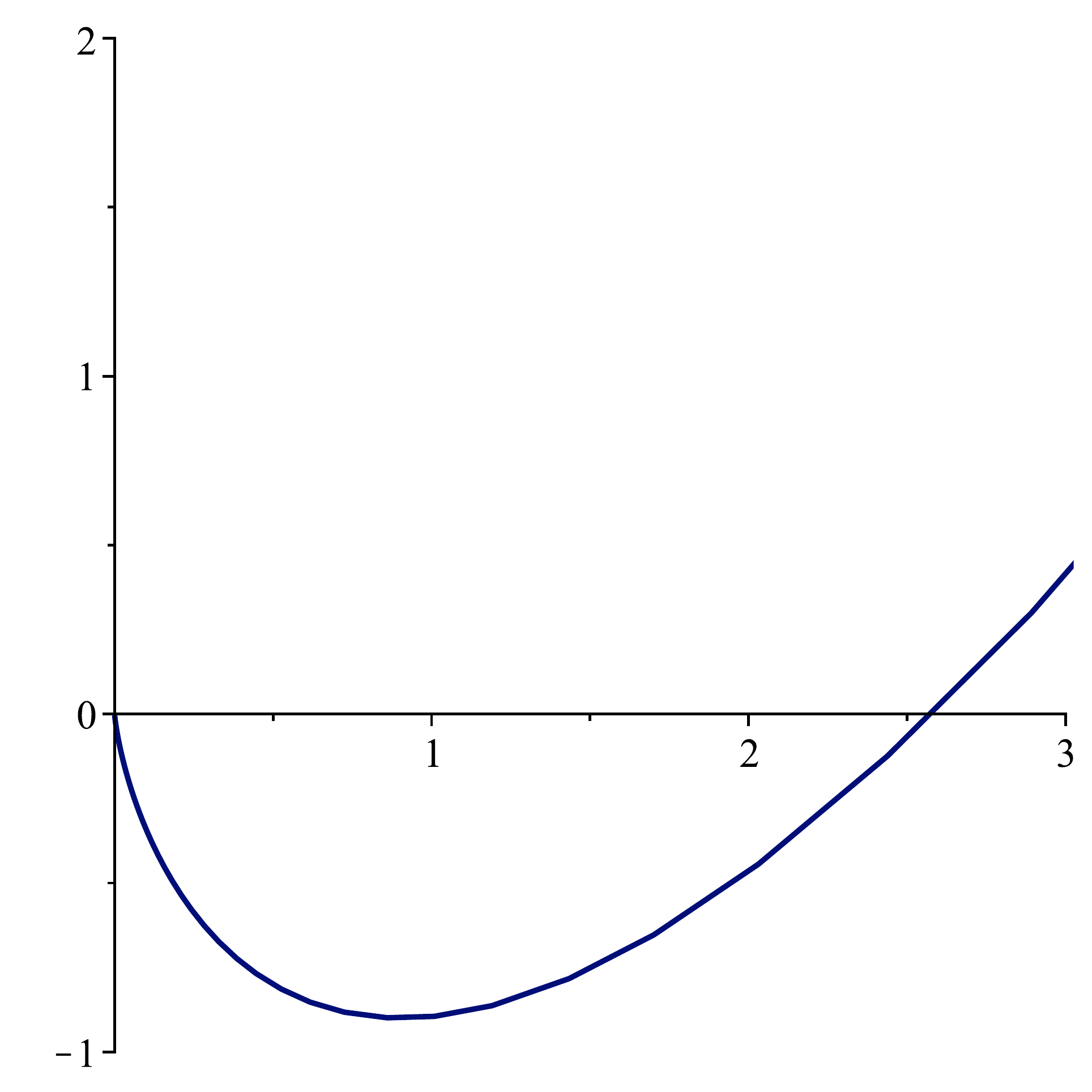}}
\fbox{\includegraphics[width=3.8cm]{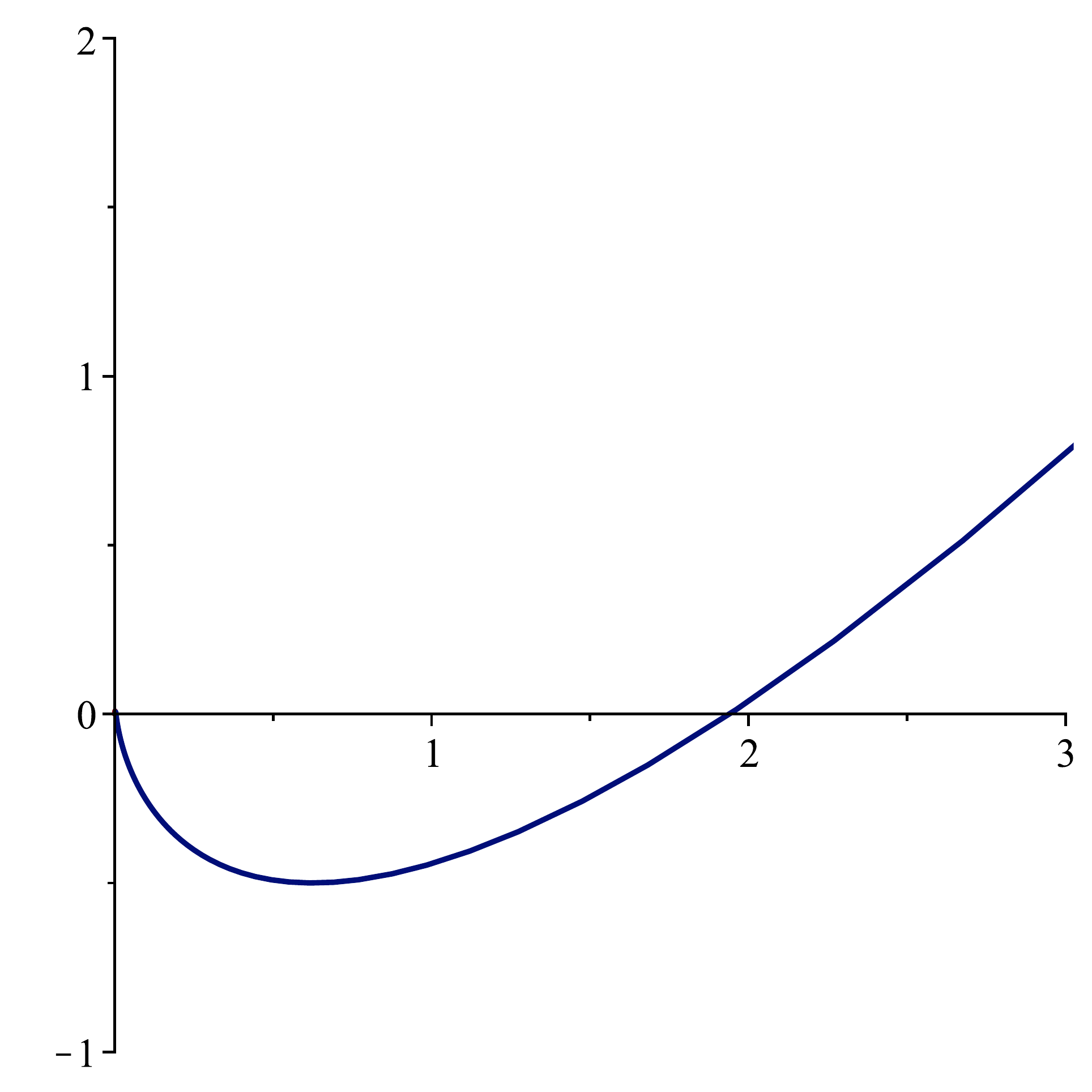}}	
\fbox{\includegraphics[width=3.8cm]{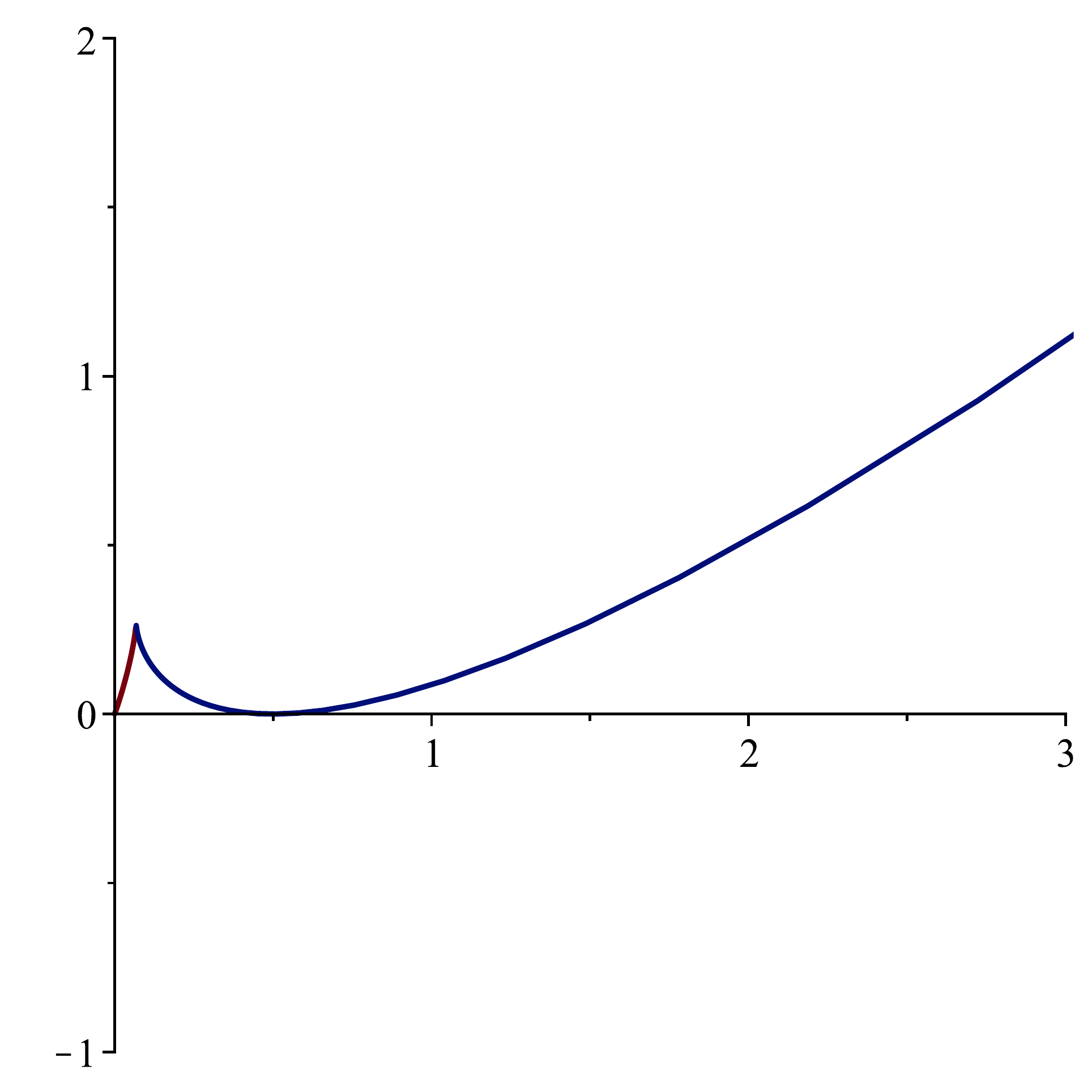}}\\
\fbox{\includegraphics[width=3.8cm]{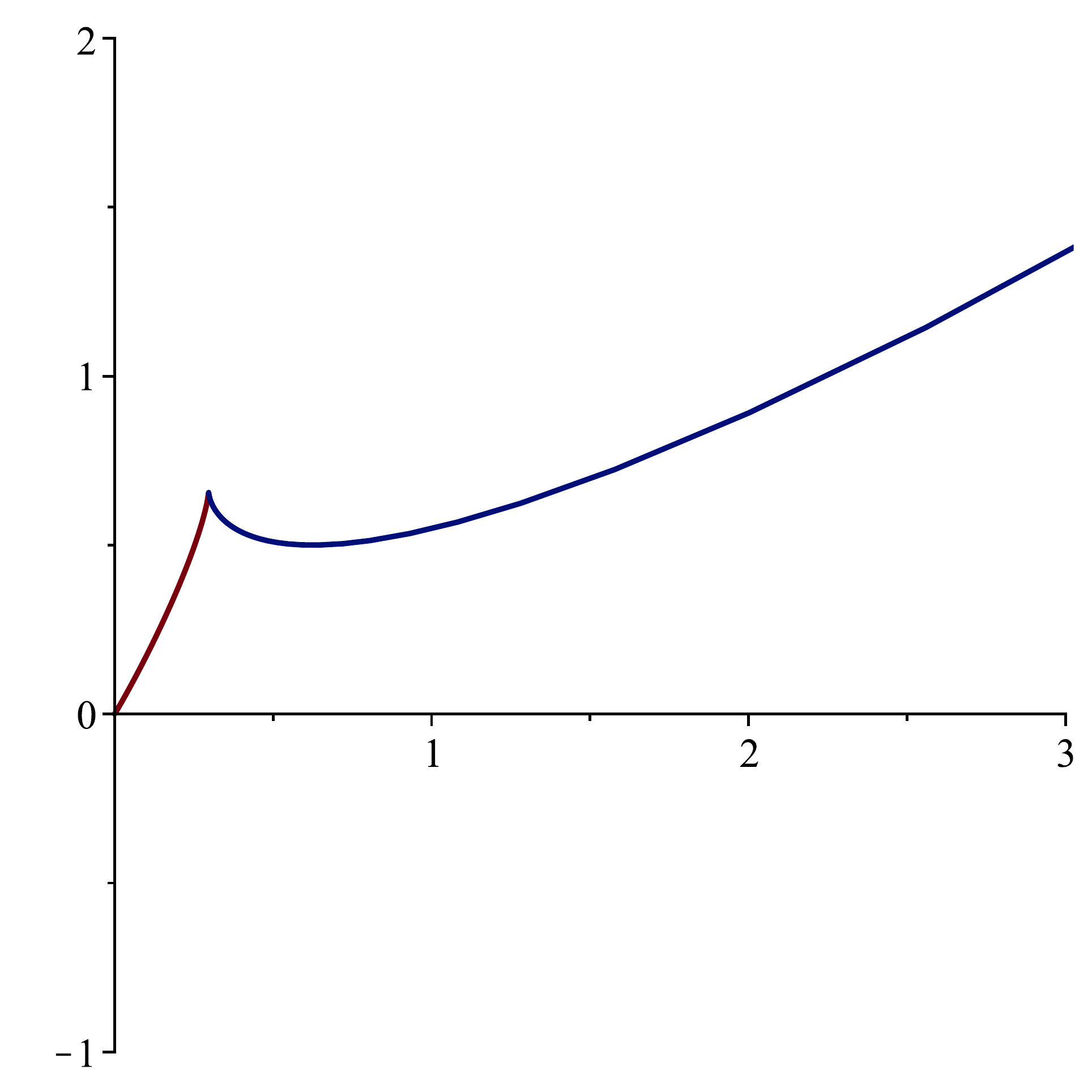}}
\fbox{\includegraphics[width=3.8cm]{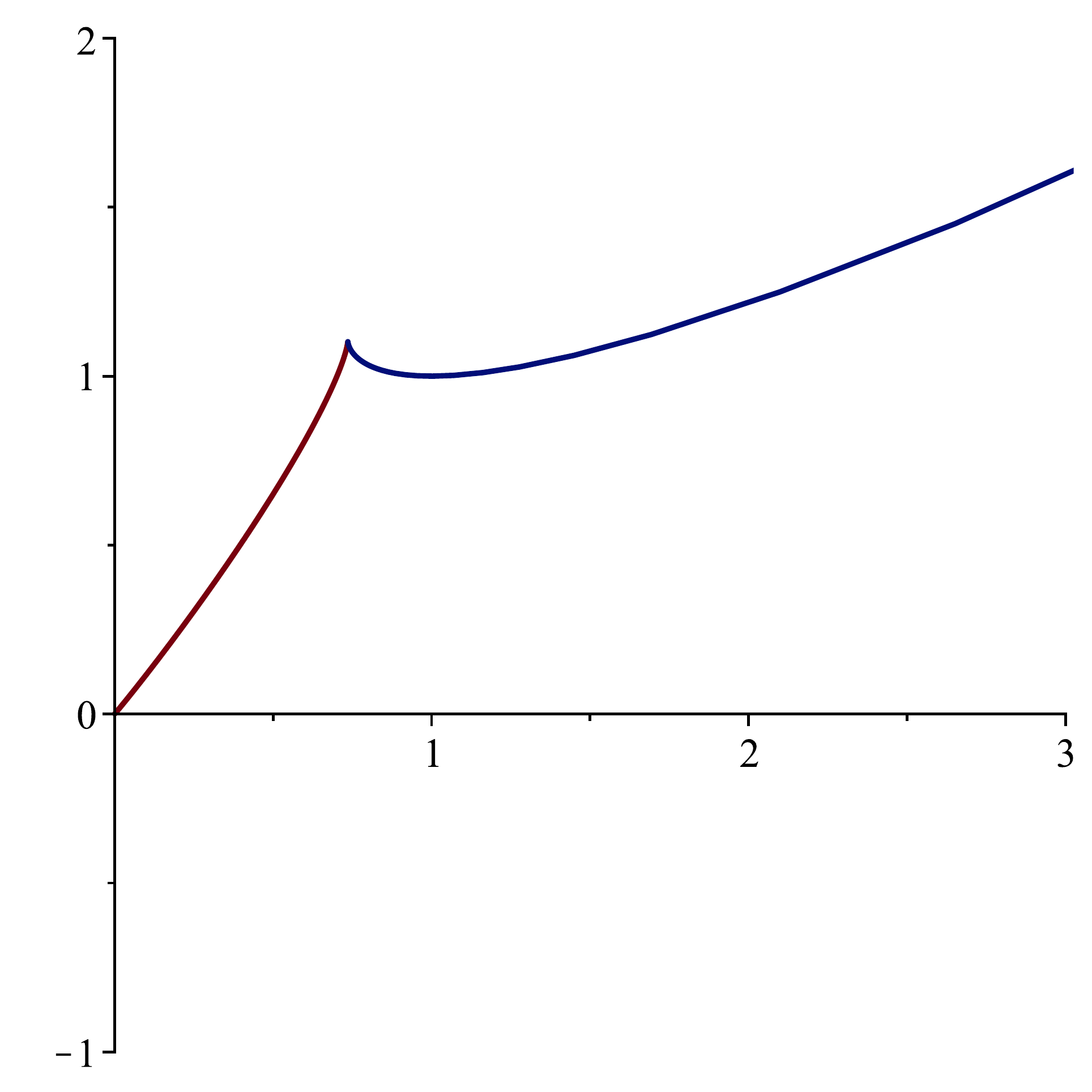}}
\fbox{\includegraphics[width=3.8cm]{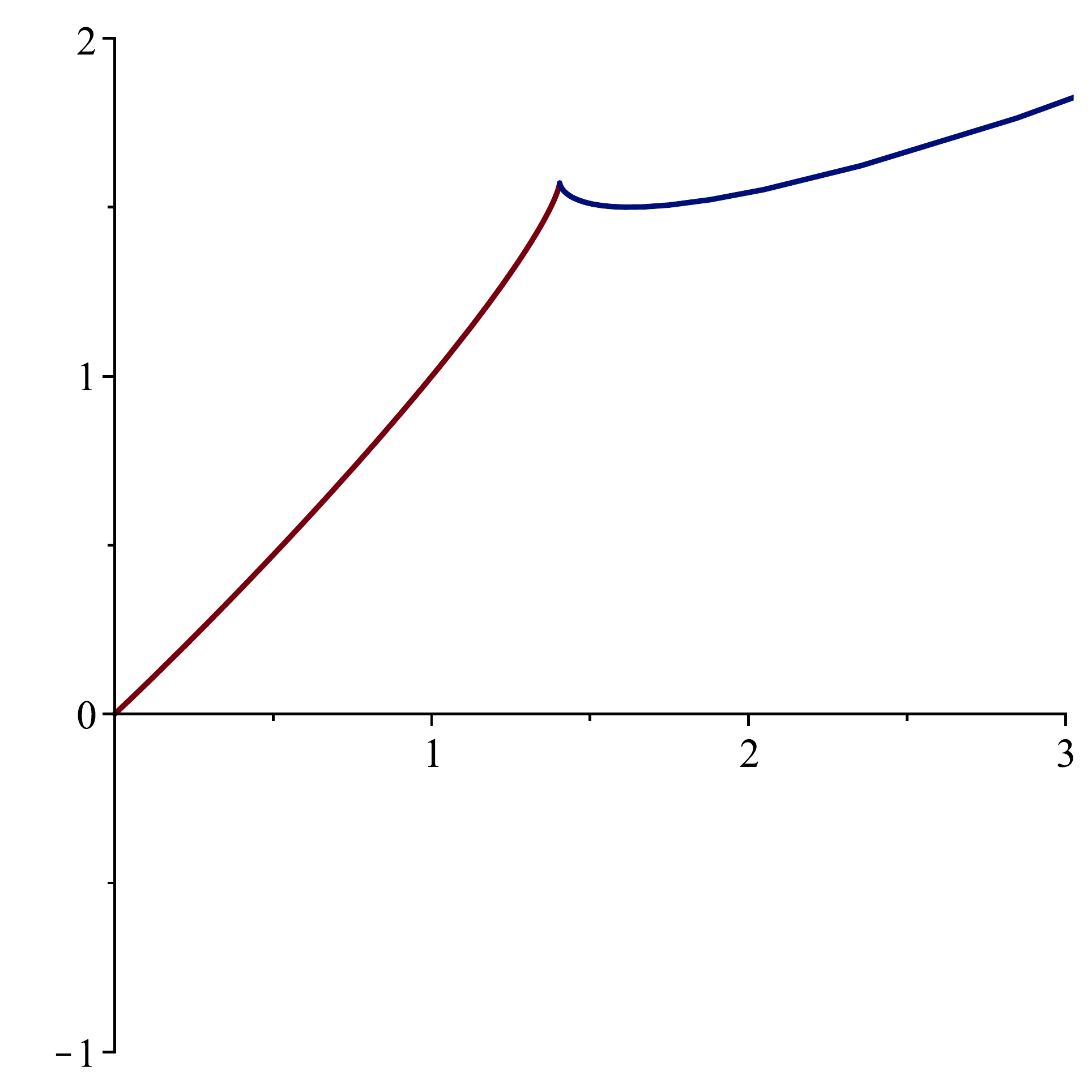}}
\caption{The solution of the Hunter-Saxton equation corresponding to $g(w)=e^w$ for $t=0,1/2,1,3/2,2,5/2$. The blue curve is parametrized by $u_x \in (-\infty,2/t)$ while the red one is parametrized by $u_x \in (2/t,\infty)$. }
\label{gexp}
\end{figure}

Consider now the solution given by $g(w)=w(1+w)(1-w)$ and $C(t)\equiv 0$. It is given by 
{\small \begin{align*}
x &=  \frac{u_x^2 (5t^4 u_x^4-60 t^3 u_x^3-4t^2 u_x^4+300 t^2 u_x^2+48 t u_x^3-640 t u_x-240 u_x^2+480) }{15(t u_x-2)^6}, \\
u &= -\frac{2 u_x^3 (5 t^3 u_x^3-60 t^2 u_x^2- 8 t u_x^3+180 tu_x+96u_x^2-160)}{15(t u_x-2)^6}.
\end{align*}}
This solution is multivalued, has nonsmooth points and, for fixed $t$, it is not defined for every $x$. See Figure \ref{star}. It is possible to eliminate $u_x$ from the two equations above to obtain a 16-degree algebraic equation in $t,x,u$ (the highest power of $u$ is 6). 
\begin{figure}[h]
\centering
\includegraphics[width=0.8\textwidth]{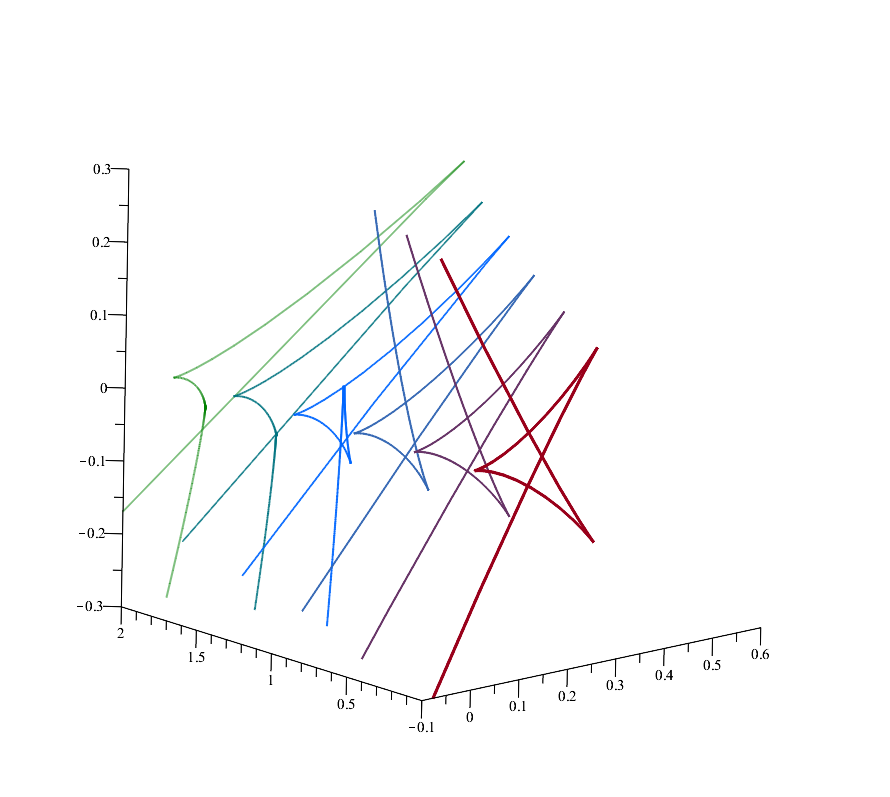}
\caption{Slices of the solution of the Hunter-Saxton equation corresponding to $g(w)=w(1+w)(1-w)$ at $t=0.0,0.4,0.8,1.2,1.6,2.0$, from red to green.  }
\label{star}
\end{figure}

\subsection{Cauchy data}
Consider the Cauchy data $u(t_0,x)=u_0(x)$. If  there exists a function $g$ such that (\ref{eq:HSsol}) satisfies the Cauchy data, it follows from  (\ref{G}) that $g$ satisfies the functional equation
\begin{equation}
 g\left(\frac{2 u_0'(x)}{2-t_0 u_0'(x)}\right) =\frac{(2-t_0 u_0'(x))^4}{16 u_0''(x)}. \label{eq:Cauchy}
\end{equation}
Notice that one necessary condition, for this equation to have a solution, is $u_0''(x)\neq 0$. 

A solution $u=\varphi(t,x)$ of the Hunter-Saxton equation, is transformed by the flow of $Y_f$ to 
\[u=\varphi(t,x+f(t))-f'(t).\] Thus, for a $g$ satisfies (\ref{Cauchy}), we must choose an appropriate function $C(t)$ to make the solution fit with the Cauchy data. Moreover, there are many solutions that satisfy the same Cauchy data.  In particular, if $f(t_0)=0$ and $f'(t_0)=0$, the transformation above will not affect the Cauchy data. Thus in order to specify the Cauchy problem completely, additional restrictions must be given, determining the function $C$ completely. Hunter and Saxton (\cite{HS}) determined $C$ by imposing the boundary condition $\lim_{x \to \infty} u(x,t) \to 0$.

Let us find a solution to the Cauchy data $u(1,x)=e^{-x}$. Solving (\ref{eq:Cauchy}) with this condition gives \[ g(w)=-\frac{8}{w(w+2)^3}.\] This gives the solution 
\begin{align*}
x &= -\frac{2(t-1)^2}{(w+2)^2}+\frac{2(t^2-1)}{w+2}-ln(-w)+ln(w+2) + C(t), \\
u &= \frac{4(1-t)}{(w+2)^2}+\frac{4t}{w+2}+C'(t).
\end{align*}
For $t=1$ we have $u(x,t) \to 0$ as $x \to \infty$, and we impose this condition for every $t$ in order to determine the constant $C(t)= -t^2/2-t+2-\ln 2$. There is a nonsmooth point on the solution, moving along the curve $2u=e^{2-x}$. Figure \ref{Cauchyexp} shows a plot of the solution at three different times. 
\begin{figure}[h]
\centering
\fbox{\includegraphics[width=3.8cm]{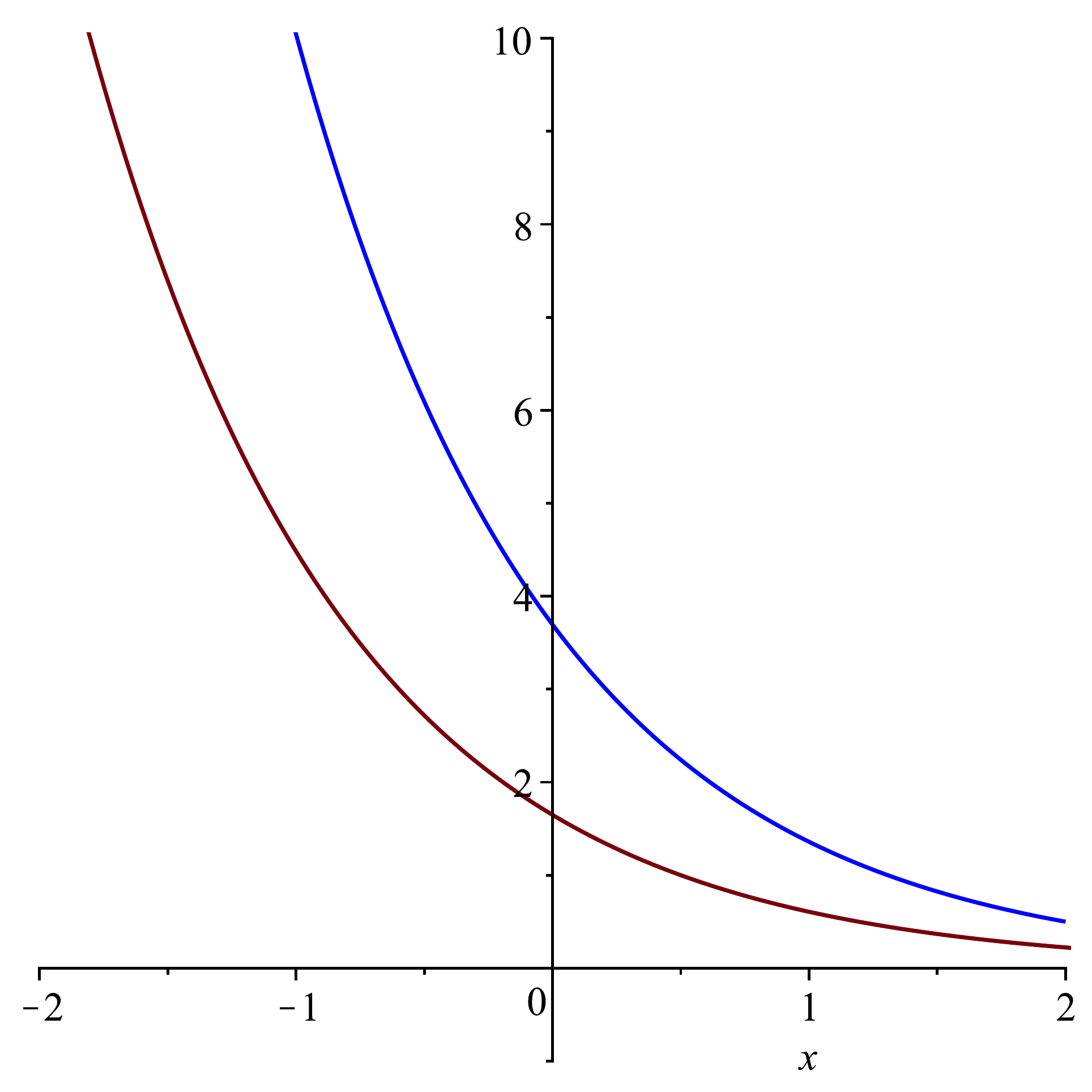}}
\fbox{\includegraphics[width=3.8cm]{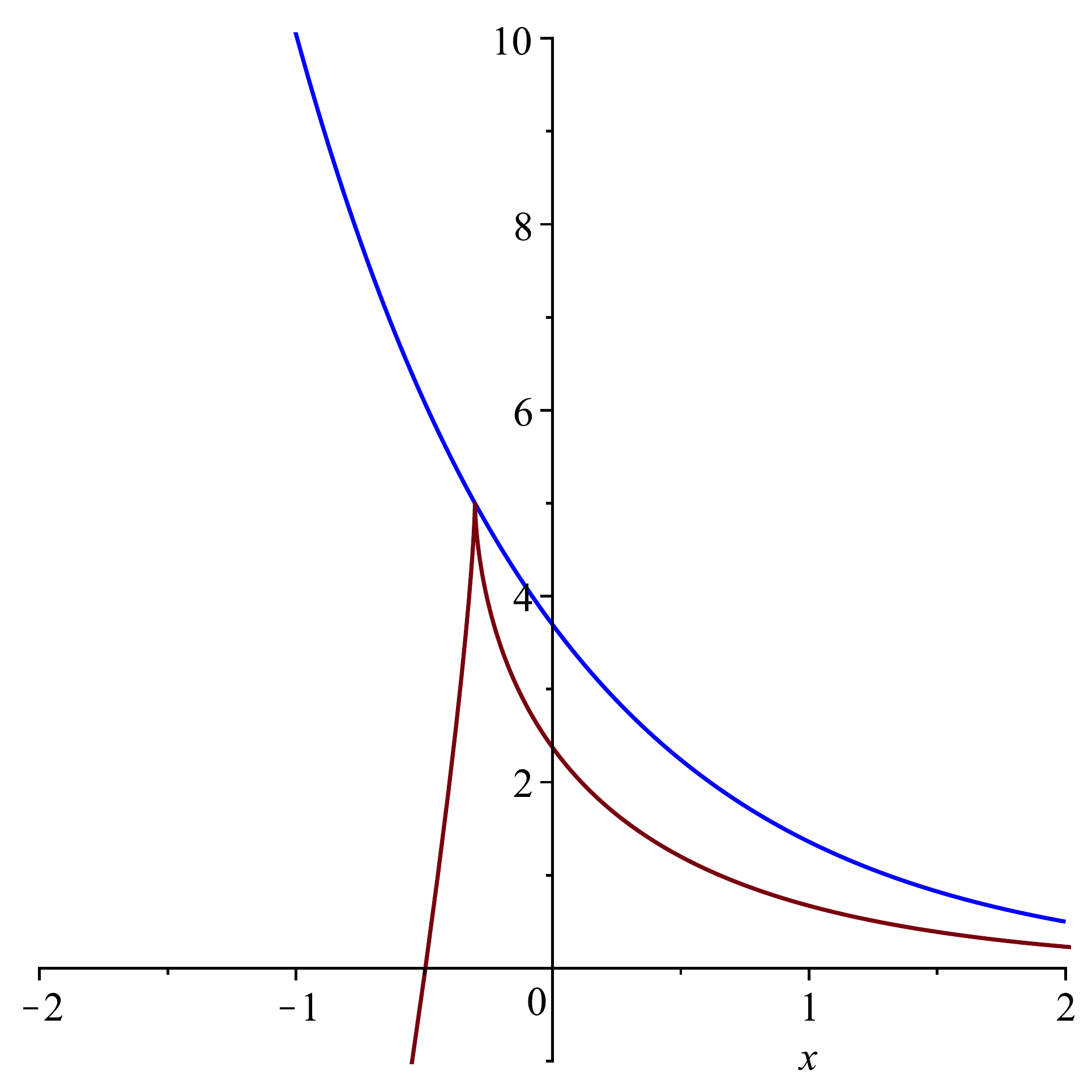}}	
\fbox{\includegraphics[width=3.8cm]{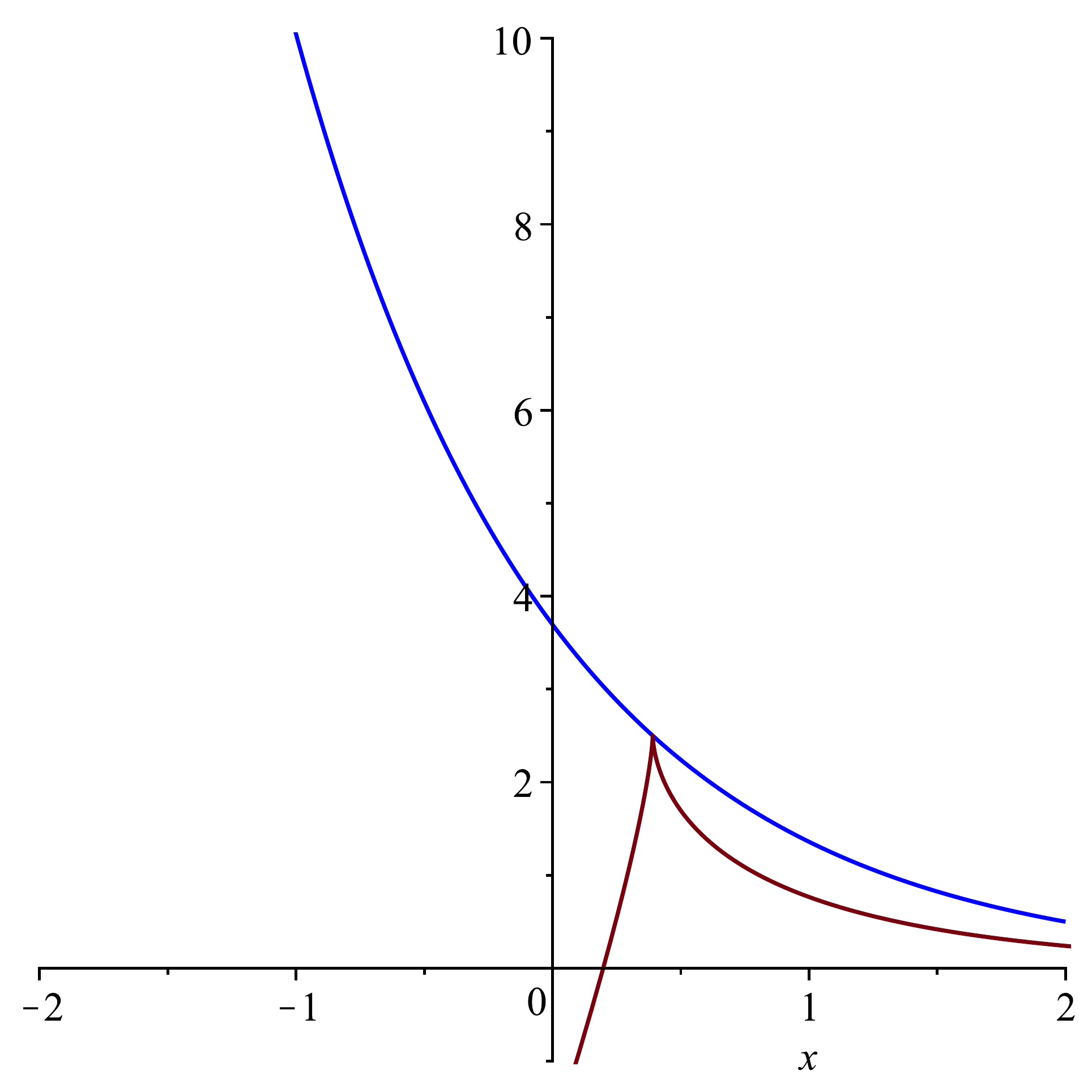}}
\caption{The curve of singularities in the $(x,u)$-plane (blue) plotted against the solution of Cauchy data $u(1,x)=e^{-x}$ at $t=1,1.2,1.4$ (red).}
\label{Cauchyexp}
\end{figure}

Let us find a solution to the Cauchy data $u_0(x)=u(1,x)=x^2$. The functional equation (\ref{eq:Cauchy}) gives $g(w)=\frac{8}{(2+w)^4}$. By integrating and eliminating $w$ from (\ref{eq:HSsol}), we get a solution with
\[u(1,x)=(x+C(1)+1)^2-1-C'(1).\] So let $C(t)=-t$ (the choice is not unique). Then we get the implicit solution 
\begin{align*}
&(t-1)^4 u^3-3 (2 t x-2 x+1) (t-1)^2 u^2+3 (2 t x-2 x+1)^2 u \\&(2-8x^3) (t-1)-3 x^2+6 (t-1)^2 x+(t-1)^4=0.
\end{align*}
If we solve for $u$, we get 
\begin{equation*}
u(t,x)=\frac{2x(t-1)+1-\left((t-1)^3+3x(t-1)+1\right)^{2/3}}{(t-1)^2}.
\end{equation*}

Considering these as curves in the $(x,u)$-plane, parametrized by $t$, we can find singular points. They are given by 
\[x= -\frac{(t^2-3t+3) t}{3(t-1)}, \qquad u=-\frac{2(t-1)^3-1}{3(t-1)^2}.\]
As $t\to 1$ we see that $(x,u) \to (\pm \infty, \infty)$. Eliminating $t$ from the equations above results in 
\[3x^2 u^2+4 x^3-u^3+1=0.\]
See Figure \ref{Cauchy}. 

\begin{figure}[h]
\centering
\fbox{\includegraphics[width=3.8cm]{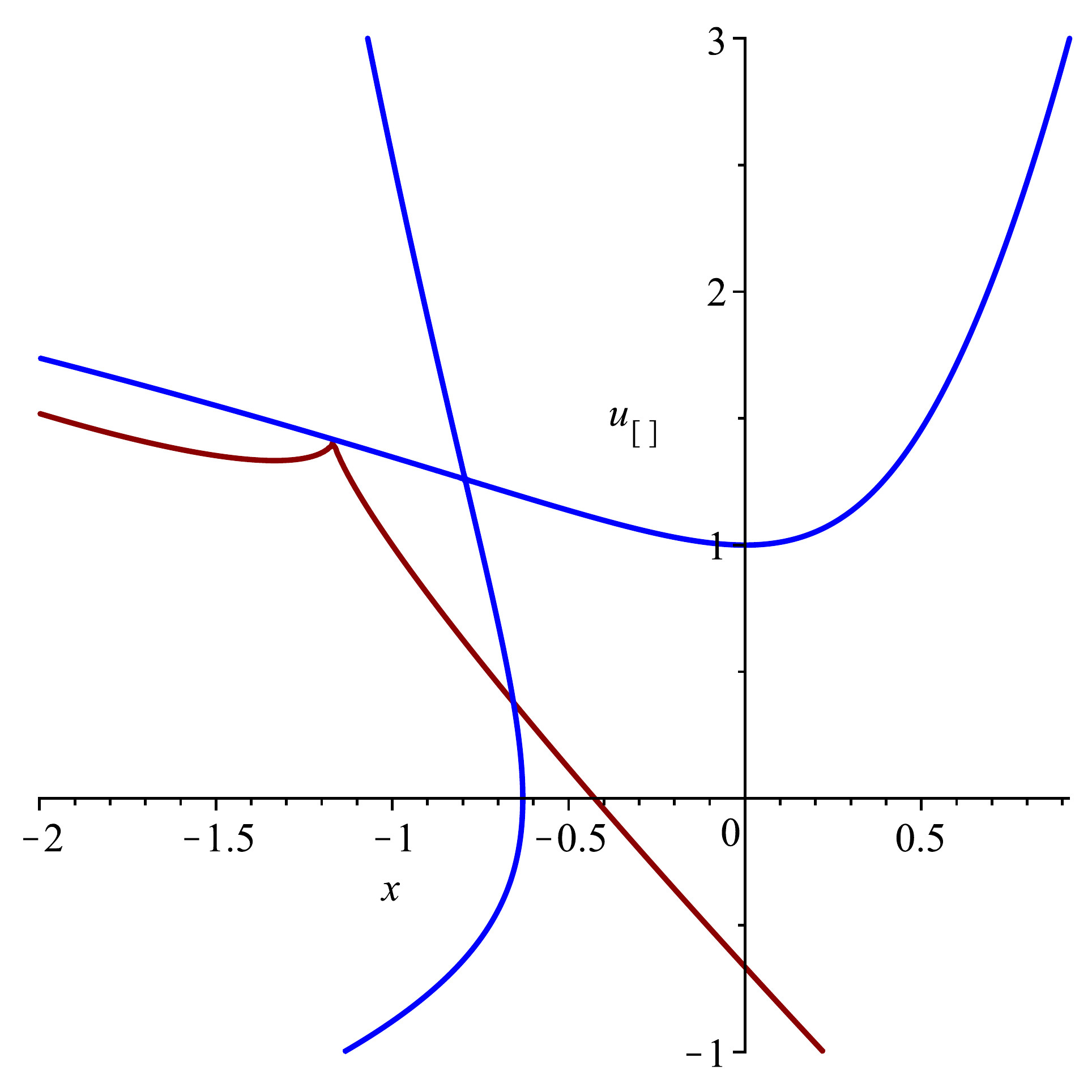}}
\fbox{\includegraphics[width=3.8cm]{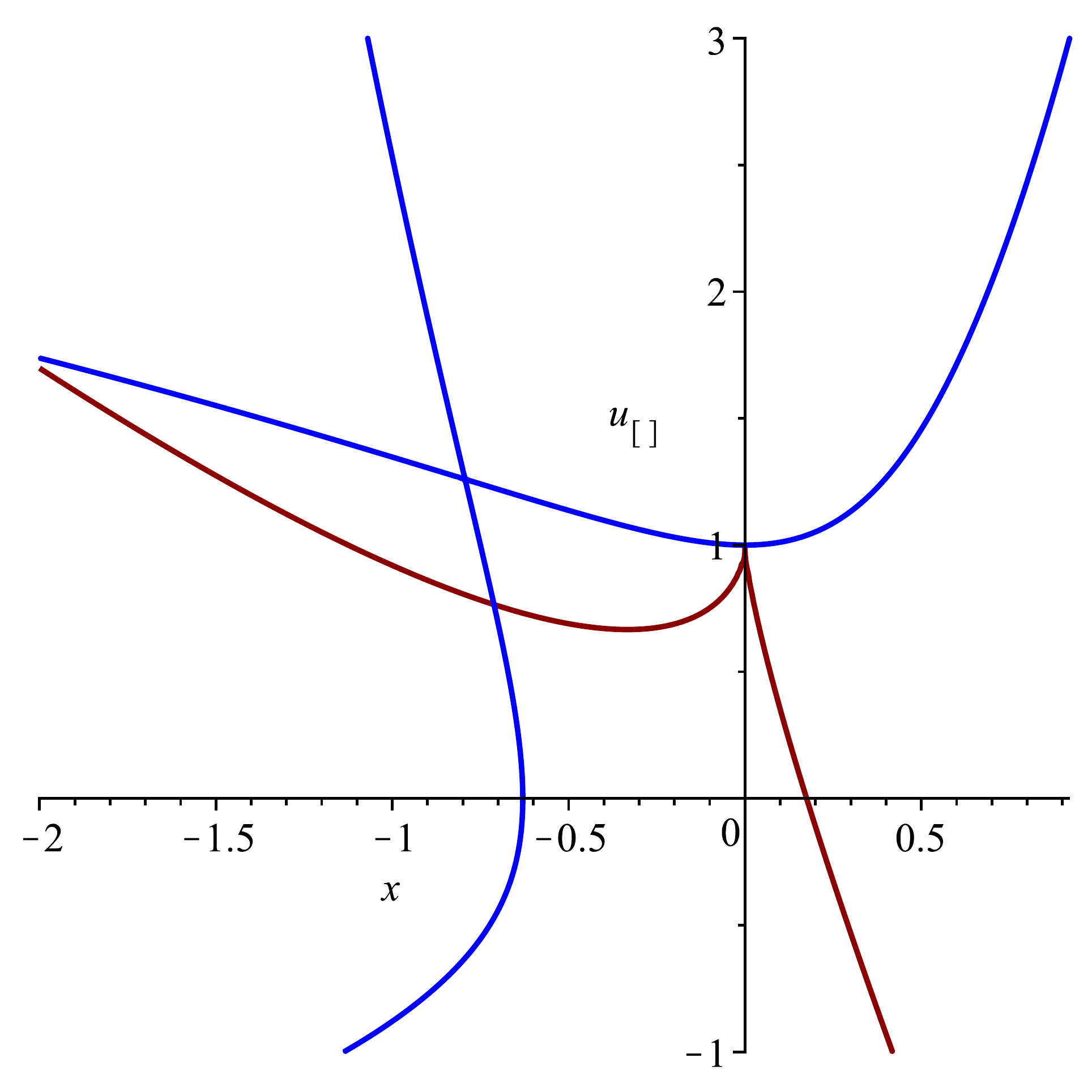}}	
\fbox{\includegraphics[width=3.8cm]{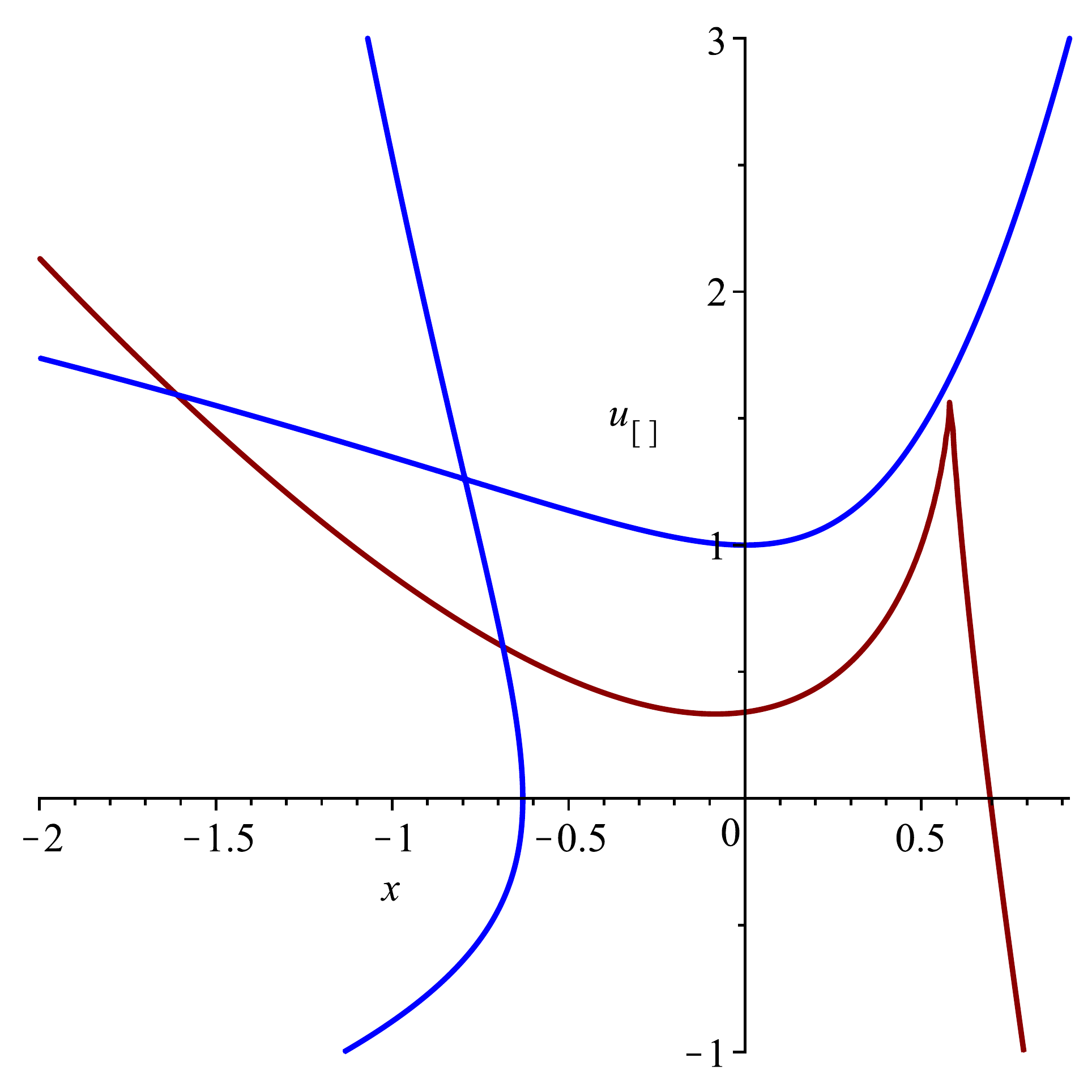}}\\
\fbox{\includegraphics[width=3.8cm]{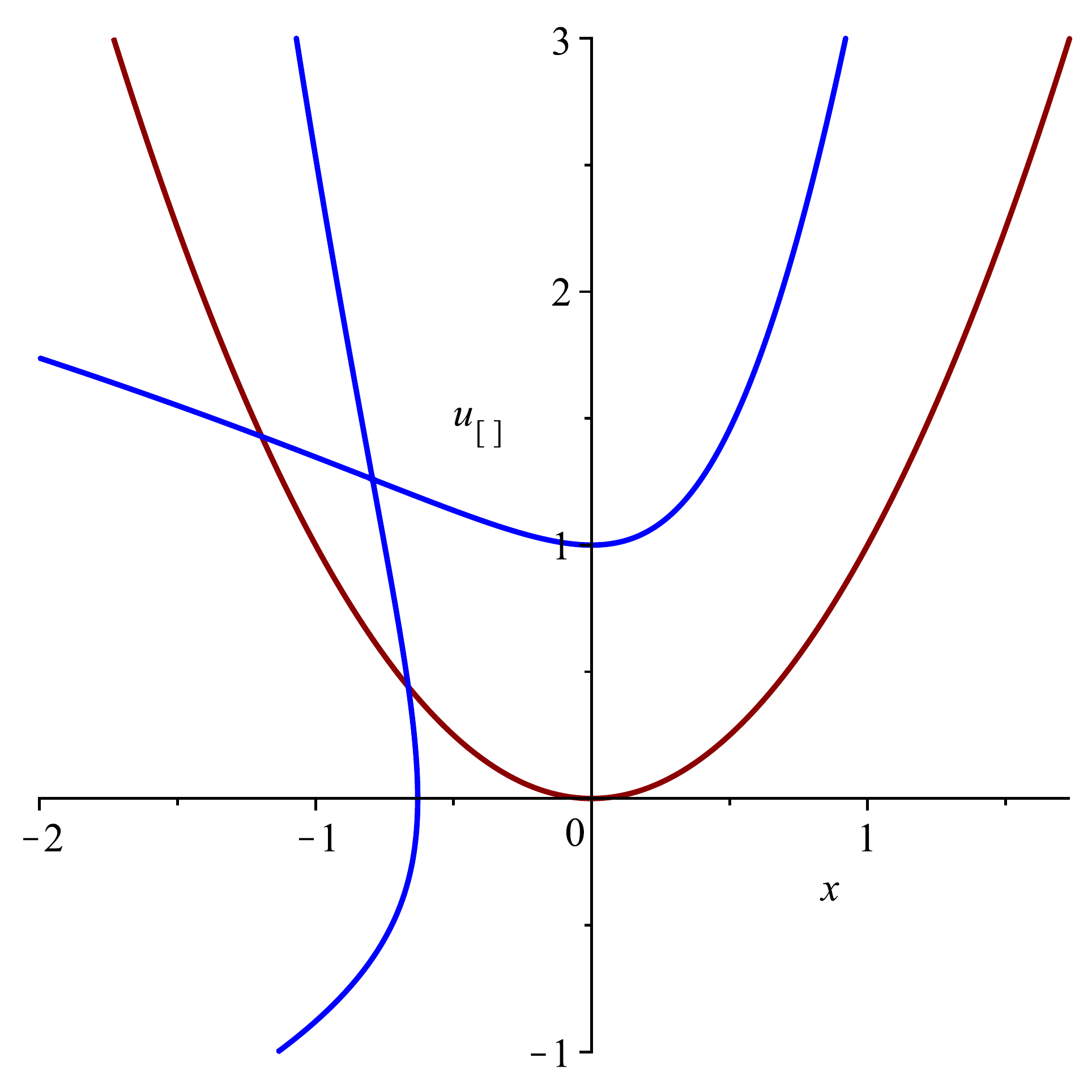}}
\fbox{\includegraphics[width=3.8cm]{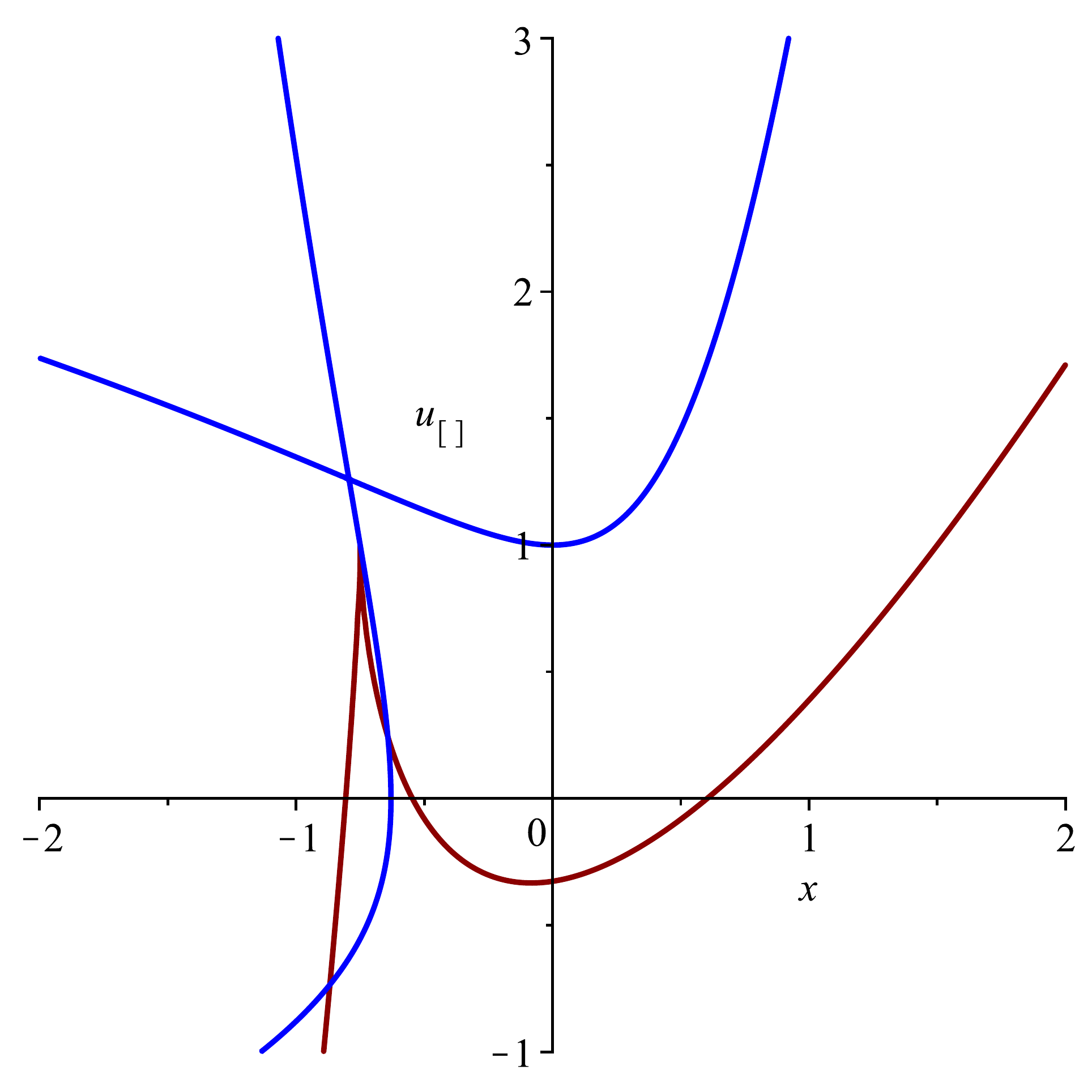}}
\fbox{\includegraphics[width=3.8cm]{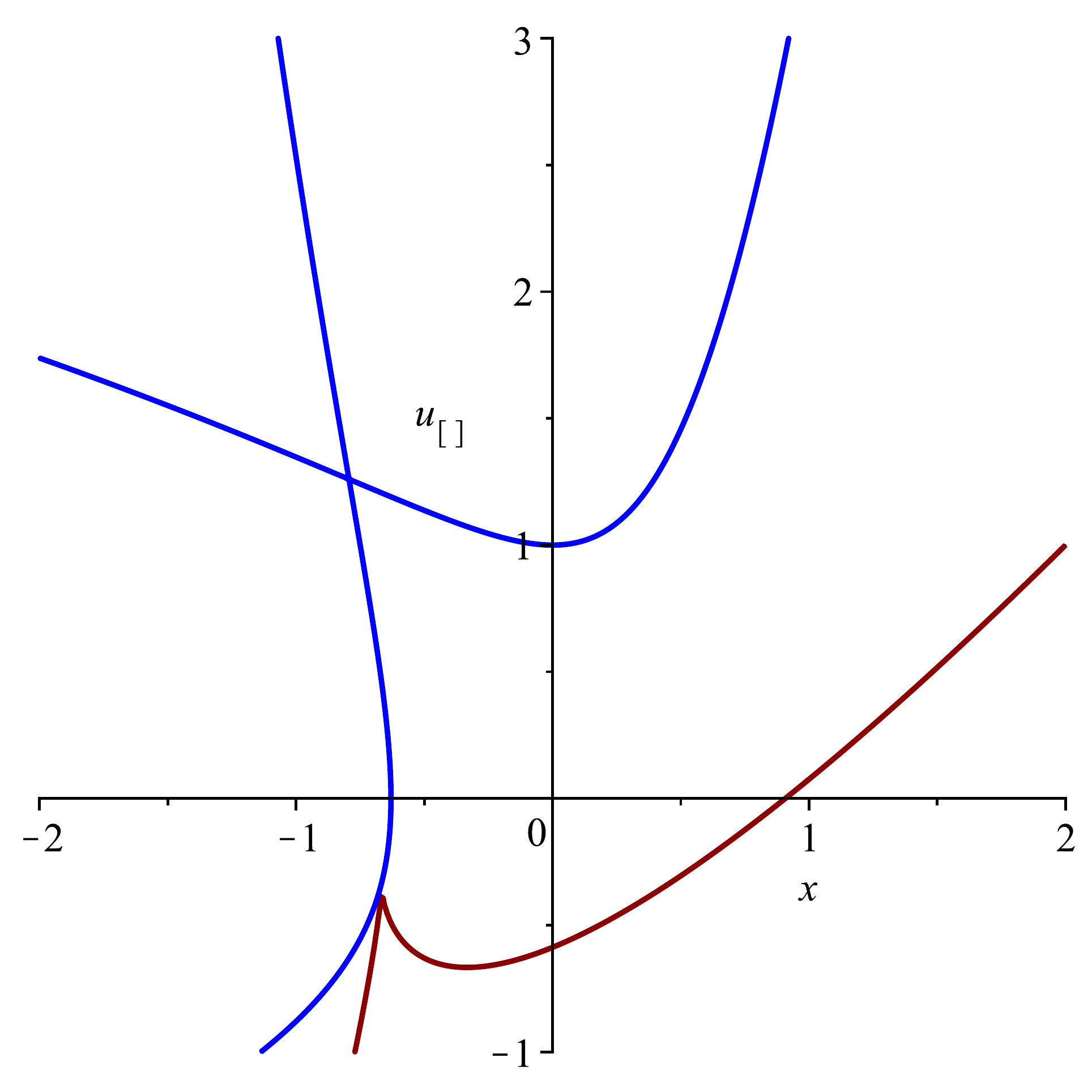}}
\caption{The curve of singularities in the $(x,u)$-plane (blue) plotted against the solution of $u(1,x)=x^2$ at $t=-1,0,1/2,1,3/2,2$ (red).}
\label{Cauchy}
\end{figure}

If we try to solve the initial value problem $u_0(x)=u(0,x)=x^2$, the method will not work. In this case we get $g(w)\equiv 1/2$, and we get the implicit solution 
\[t^4 u^3-6t^3 xu^2+12 t^2 x^2u-8tx^3+9x^2=0.  \]
We notice that for $t=0$, this gives $9x^2=0$. 

The initial time $t=0$ is special due to our choice of $G$, so we should solve the initial value problem away from this point. Note that this issue does not reflect a property of the Hunter-Saxton equation since it has time-translations as symmetries. In particular, we may translate the solution with $u(1,x)=x^2$, to a solution satisfying $u(0,x)=x^2$. This solution is given by 
\[u(t,x)=\frac{1+2tx-(t^3+3tx+1)^{2/3}}{t^2}.\]

\subsection{The action of the remaining symmetries on the quotient}
We chose a particular Lie subalgebra of the symmetry Lie algebra of the Hunter-Saxton equation, and computed the quotient PDE. Each solution of the quotient is given by a function $g$. The PDE $\{F=0,G=0\}$ is symmetric with respect to the vector fields $Y_f$, but not with respect to the rest of the symmetries of the Hunter-Saxton equation. The flows of the remaining symmetries will act on the function $g$. Intuitively, they shuffle equivalence classes of solutions. 
\begin{align*}
\partial_t \qquad &\leftrightarrow \qquad g(w) \mapsto \frac{16 g\left(\frac{2w}{2-sw}\right)}{(2-sw)^4} \\
t \partial_t-x \partial_x-2 u \partial_u \qquad &\leftrightarrow \qquad g(w) \mapsto g(e^{-s}w) \\
x\partial_x+u\partial_u \qquad & \leftrightarrow \qquad g(w) \mapsto e^{-s} g(w)\\ 
t^2 \partial_t+2tx\partial_x+2x\partial_u \qquad & \leftrightarrow \qquad g(w) \mapsto g(w+2s)
\end{align*}


\section{Solving PDEs with first order quotient} \label{Main}
By considering several different second-order PDEs with infinite-dimensional symmetry Lie algebra, we illustrate the applicability of the ideas explained above.   The PDEs we consider will be invariant under one of four different infinite-dimensional Lie algebras of the form $\mathfrak g=\langle X_f \mid f(t) \in C^{\infty}_{\text{loc}}(\mathbb R) \rangle$,  with generators
\begin{enumerate}
 \item $X_f= f(t) \partial_x+f'(t) \partial_u $,
 \item $X_f = f(t) \partial_t-f'(t) \partial_u$,
 \item $X_f= f(t) \partial_t$,
 \item $X_f =  f(t) \partial_u$.
\end{enumerate}
They are considered in sections \ref{sym1}, \ref{sym2}, \ref{sym3} and \ref{sym4}, respectively. Notice that all these vector fields are projectable to $\mathbb R^2(t,x)$, so their flow takes graphs of functions (on $\mathbb R^2 \times \mathbb R$) to graphs of functions. 

For each of these infinite-dimensional Lie algebras we will find the general invariant PDE of the form
\[F=u_{tx}-\varphi(t,x,u,u_t,u_x,u_{tt},u_{xx})=0.\] 
Notice that PDEs for which $\partial_{u_{tx}}(F)=0$ holds have, for all our Lie algebras, the additional properties $\partial_{u_{t}}(F)=0$, $\partial_{u_{tt}}(F)=0$.  They are essentially ODEs, possibly parametrized by $t$.

In order for the global Lie-Tresse theorem (\cite{KL2}) to apply, we must require the fibers of $\mathcal E_k \to  J^0(\mathbb R^2)$ to be (irreducible) algebraic manifolds. This restricts $\varphi$ even more. We assume that the equations in this chapter satisfy this condition.

We will find the differential invariants and quotient PDE for the general symmetric PDE, and then take a closer look at more specific PDEs. 
We will focus on computations, with the aim of getting a feeling for how the ideas explained above can be used efficiently. For some PDEs we write down the general solution, and for others we will be satisfied with only solving the quotient PDE.

In Section \ref{finitesym} we discuss the ideas from a slightly different angle. We realize that we can consider finite-dimensional Lie subalgebras of $\mathfrak g$ to obtain two first-order syzygies, as we did for Burgers' equation in Section \ref{Burgers}. But because the Lie algebra in these other cases is a subalgebra of infinite-dimensional ones, these syzygies are partially uncoupled.

\subsection{Symmetries of type 1} \label{sym1}
Consider the Lie algebra spanned by vector fields of the form $f(t) \partial_x+f'(t) \partial_u$. The equation \[u_{tx}=\varphi(t,x,u,u_t,u_x,u_{tt},u_{xx})\] is invariant if and only if $\varphi=-\alpha(t,u_x,u_{xx})-uu_{xx}$. Thus we consider the PDE
\[F=u_{tx}+uu_{xx}+\alpha(t,u_x,u_{xx}).\]
When $\alpha=u_x^2/2$ we get the Hunter-Saxton equation, considered above. The algebra of differential invariants is generated by 
\[I=t, \qquad J=u_x, \qquad H=u_{xx}.\]
The quotient PDE is given by 
\[\hat \partial_I(H)-(\alpha-H \alpha_H) \hat \partial_J(H)+(J+\alpha_J)H=0. \]

\textbf{Example 1.1:} Consider the PDE given by $u_{tx}+uu_{xx}+\alpha(u_x)=0$. A special instance of this is the ``generalized Hunter-Saxton equation'', with $\alpha(u_x)=\epsilon u_x^2$, which was considered in \cite{Pavlov} and \cite{MorozovHS}. Its quotient PDE is given by $\hat \partial_I(H)-\alpha(J) \hat \partial_J(H)+(J+\alpha'(J)) H=0$ which has general solution 
\[Hg\left(\int \frac{dJ}{\alpha(J)} +I\right)= e^{\int \frac{J+\alpha'(J)}{\alpha(J)} dJ}.\] This gives the additional equation \[u_{xx} g\left(\int \frac{du_x}{\alpha(u_x)} +t\right)= e^{\int \frac{u_x+\alpha'(u_x)}{\alpha(u_x)} du_x} \] which can be considered as a first-order PDE in $u_x$. Its solution is given implicitly by 
\[x= \int \frac{g\left(\int \frac{du_x}{\alpha(u_x)} +t\right)}{e^{\int \frac{u_x+\alpha'(u_x)}{\alpha(u_x)} du_x}} du_x+C(t).\] By solving the original PDE for $u$, we get 
\[u=-\frac{u_{tx}-\alpha(u_x)}{u_{xx}},\] where $u_{tx}$ and $u_{xx}$ may be replaced by functions of $t$ and $u_x$ (in the same way as for the Hunter-Saxton equation). We end up with a solution parametrized by $t$ and $u_x$. 

\textbf{Example 1.2:} Consider the PDE $u_{tx}+uu_{xx}+\alpha(t,u_x) u_{xx} =0$. The quotient PDE is given by $\hat \partial_I(H)+(J+\alpha_J H)H=0$. It has general solution 
\[H=\frac{e^{-IJ}}{g(J)+\int \alpha_J(I,J)e^{-IJ} dI}.\] By considering this as a first-order PDE on $u_x$ we get the implicit solution 
\[x= \int e^{t u_x} \left(g(u_x)+ \int e^{-tu_x}\alpha_{u_x}(t,u_x) dt  \right) du_x +C(t).\]
We also have
\[u=-\frac{u_{tx}}{u_{xx}}-\alpha(t,u_x),\] where $u_{tx}$ and $u_{xx}$ can be eliminated to give a parametrization of the solution by $t$ and $u_x$. 

\textbf{Example 1.3:} Consider the PDE $u_{tx}+uu_{xx}+u_{xx}^2=0$. The quotient PDE is given by $\hat \partial_I(H)+H^2 \hat \partial_J(H)+JH=0$ whose general solution is given implicitly by 
\[\left(g(J^2+H^2)-I \right) \sqrt{J^2+H^2}+\text{arctanh}\left(\frac{J}{\sqrt{J^2+H^2}} \right)=0. \] Again, this is a first-order PDE on $u_x$, but contrary to the previous cases this PDE can not be solved easily as a first-order, separable ODE.

\subsection{Symmetries of type 2} \label{sym2}
Consider the Lie algebra spanned by $f(t) \partial_t-f'(t) \partial_u$. The general invariant second-order PDE, assuming it can be solved for $u_{tx}$, is of the form
\[u_{tx}-\alpha(x,u_x,u_{xx}) e^u=0.\]
The differential invariants are generated by 
\[I=x,\quad J=u_x, \quad H=u_{xx},\qquad \hat \partial_I=-\frac{u_{xx}}{u_{tx}} D_t+D_x, \quad \hat \partial_J=\frac{1}{u_{tx}} D_t\] 
and the quotient PDE is given by 
\[\alpha_H \hat \partial_I(H)+\left(H \alpha_H-\alpha \right) \hat \partial_J(H)+\alpha_J H+\alpha J+\alpha_I=0.\]
Let $H=g(I,J)$ be a solution of this equation. It gives an equation $u_{xx}=g(x,u_x)$, which can be viewed as a first-order PDE on $u_x$. Now, let $u_x(t,x)=w(t,x)$ be a solution of this PDE. Inserting it into $u_{tx}=\alpha e^u$ results in the solution
\[u(t,x)= \ln \left( \frac{w_t(t,x) }{\alpha(x,w(t,x),w_x(t,x))}\right).\]
Let us consider a few different choices of $\alpha$. 

\textbf{Example 2.1:} Assume that $\alpha=H \beta(I,J)$. The quotient PDE now takes the form 
\[\beta \hat \partial_I(H)+\beta_J H^2+\beta J H +\beta_I H=0\]
and has general solution 
\[H e^{IJ} \beta(I,J) \left(\int \frac{\beta_J(I,J)}{e^{IJ} \beta(I,J)^2} dI +g(J)\right) = 1.\]
This gives 
\[u_{xx} e^{x u_x} \beta(x,u_x) \left(\int \frac{\beta_{u_x}(x,u_x)}{e^{x u_x} \beta(x,u_x)^2} dx +g(u_x)\right) = 1.\]

\textbf{Example 2.2:} Let us now assume that $\alpha=\alpha(I,J)$. Then the quotient reduces to 
\[-\alpha \hat \partial_J(H)+\alpha_J H+ \alpha J+\alpha_I=0\]
which has general solution 
\[H=\left( \int \frac{J \alpha(I,J)+\alpha_I(I,J)}{\alpha(I,J)^2} dJ+g(I)\right) \alpha(I,J).\]  Inserting the expressions for $I,J,K$ gives 
\[u_{xx}=\left( \int \frac{u_x \alpha(x,u_x)+\alpha_x(x,u_x)}{\alpha(x,u_x)^2} du_x+g(x)\right) \alpha(x,u_x). \] 

\textbf{Example 2.3:} We continue the computations here for $\alpha \equiv -1$. In this case the quotient is given by
\[\hat \partial_J(H)-J=0.\] 
Its general solution is 
\[ H=\frac{1}{2} J^2+ g(I)\]
which gives the differential constraint
\[ u_{xx}= \frac{1}{2} u_x^2 + g(x).\]
Interpreted as a first-order ODE on $u_x$ this is a Riccati equation. The transformation $u_x(t,x)=-2 v_x(t,x)/v(t,x)$ gives 
\[ 2\frac{-v_{xx} v+(v_x)^2}{v^2}=2 \frac{(v_x)^2}{v^2}+g(x)\] 
and 
\begin{equation}
2 v_{xx}+g(x)v=0. \label{eq:Shrod}
\end{equation}
Thus the problem of solving $u_{tx}+e^u=0$ is reduced to solving a first-order linear ODE. We refer to \cite{LychaginIntegrability} for a way to find explicit solutions to the Shrödinger type equation (\ref{eq:Shrod}).

\subsection{Symmetries of type 3} \label{sym3}
Consider the Lie algebra spanned by $f(t) \partial_t$. Assuming that the PDE can be solved for $u_{tx}$, the general invariant second-order PDE is given by 
\begin{equation}
u_{tx}-u_t \alpha(x,u,u_x,u_{xx})=0. \label{eq:ScalingsEq}
\end{equation}
The differential invariants are generated by 
\[I=x,\qquad J=u, \qquad H=u_x,\qquad \hat \partial_I=-\frac{u_x}{u_t} D_t+D_x, \qquad \hat \partial_J=\frac{1}{u_t} D_t\] 
and the quotient PDE is given by 
\[H_J-\alpha(I,J,H,H_I+H H_J)=0.\]

Assume that we have a solution of the form $H=g(I,J)$. This gives an equation $u_x=g(x,u)$, of which (\ref{eq:ScalingsEq}) is a differential consequence. Thus, solving the second-order PDE (\ref{eq:ScalingsEq}) amounts to solving, sequentially, two first-order PDEs. 

We focus on a few different choices of $\alpha$. 

\textbf{Example 3.1:} If $\alpha=\alpha_1(x,u) u_x+\alpha_2(x,u)$, the quotient equation of (\ref{eq:ScalingsEq}) is 
\[\hat \partial_J(H)-\alpha_1(I,J) H-\alpha_2(I,J)=0 \]  and its general solution is
\[H=\left(\int \alpha_2(I,J) e^{-\int \alpha_1(I,J)dJ} dJ+g(I)    \right) e^{\int \alpha_1(I,J)dJ}.\]
Inserting the expressions for the invariants gives the PDE
\[u_x=\left(\int \alpha_2(x,u) e^{\int \alpha_1(x,u)du} du+g(x)    \right) e^{\int \alpha_1(x,u)du}.\]

\textbf{Example 3.2:} Consider the PDE $u_{tx}=u_t(\alpha_1(x) u+\alpha_2(x))$. Its quotient PDE is $\hat \partial_J(H)=\alpha_1(I) H+\alpha_2(I)$. The quotient's general solution is $H=\frac{1}{2} \alpha_1(I) J^2+\alpha_2(I) J + g(I)$, or $u_x=\frac{1}{2} \alpha_1(x) u^2+\alpha_2(x) u +g(x)$. This is a Riccati equation. Choosing the the set of solutions for which $g \equiv 0$ lets us write them down explicitly:
 \[u(t,x) = \frac{2 e^{\int \alpha_2(x) dx}}{C(t)-\int \alpha_1(x) e^{\int \alpha_2(x) dx} dx}\]
 
In the case when $\alpha_1\equiv 0$ and the PDE is linear, we are able to write down the general solution 
 \[u(t,x)=\left(\int g(x) e^{-\int\alpha_2(x) dx} dx+C(t) \right) e^{\int \alpha_2(x) dx}.\]
 
\textbf{Example 3.3:} Consider the equation $u_{tx}=u_t u_x$. Its quotient is $\hat \partial_J(H)=H$. This gives $H=g(I) e^J$, or $u_x=g(x) e^u$. Thus, the general solution of $u_{tx}=u_t u_x$ is \[u(t,x)= \ln\left(\frac{1}{C(t)-\int g(x) dx}\right).\]

\subsection{Symmetries of type 4} \label{sym4}
Consider the Lie algebra spanned by $f(t) \partial_u$. Assuming the PDE can be solved for $u_{tx}$, the general invariant second-order PDE is given by 
\begin{equation}
u_{tx}=\alpha(t,x,u_x,u_{xx}). \label{eq:TranslationsEq}
\end{equation}
Note that this is a first-order PDE in $u_x$. The differential invariants are generated by 
\[ I=t, \qquad J=x, \qquad H= u_x, \qquad \hat \partial_I= D_t, \qquad \hat \partial_J = D_x\]
and the quotient PDE is given by 
\[\hat \partial_I(H) =\alpha(I,J,H, \hat \partial_J(H)).\]
The quotient is exactly (\ref{eq:TranslationsEq}) treated as a first-order PDE on $u_x$.
\begin{remark}
This shows that all first-order scalar PDEs can be obtained as a quotient of a second-order PDE. 
\end{remark}
Assume that a solution can be written as $H=g(I,J)$ for some function $g$. This gives us an equation $u_x=g(t,x)$ which can be added to (\ref{eq:TranslationsEq}), and in fact, (\ref{eq:TranslationsEq}) is just a differential consequence of this first-order PDE. The function $u(t,x)= \int g(t,x) dx+C(t)$ will be a solution to the original equation. We solve some concrete examples. 

\textbf{Example 4.1:} Consider the PDE $u_{tx}=u_x^A$, with constant $A \neq 1$. The quotient PDE is given by $\hat \partial_I(H)=H^A$ and has general solution $H=(g(J)+(1-A) I)^{1/(1-A)}$. This gives the PDE $u_x=(g(x)+(1-A) t)^{1/(1-A)}$ which is integrated to \[u(t,x)= \int \big(g(x)+(1-A) t\big)^{\frac{1}{1-A}} dx +C(t).\]

\textbf{Example 4.2:} Consider the PDE $u_{tx}=u_x^A u_{xx}$. The quotient PDE is given by $\hat \partial_I(H)=H^A \hat \partial_J(H)$ and has general solution $J+I H^A-g(H)=0$. This gives the PDE $x+t u_x^A-g(u_x)=0$. Solving for $u_x$ and integrating gives the general solution from the equivalence class determined by $g$. 

\textbf{Example 4.3:} Consider the PDE $u_{tx}=\alpha(t,x) u_x^2+\beta(t,x) u_x+\gamma(t,x)$. The quotient is given by $\hat \partial_I(H)=\alpha(I,J) H^2+\beta(I,J) H+\gamma(I,J)$, a Riccati equation. If $\gamma \equiv 0$, then \[H=\frac{e^{\int\beta(I,J) dI}}{g(J)-\int \alpha(I,J) e^{\int \beta(I,J) dI}dI }.\] Solving this PDE gives the general solution:
\[u(t,x)= \int \left(\frac{e^{\int\beta(t,x) dI}}{g(x)-\int \alpha(t,x) e^{\int \beta(t,x) dt}dt } \right) dx+C(t) \]

Some of the equations we solve here may look too trivial to be worth considering. Even though a part of their simplicity is a consequence of their symmetry Lie algebra, one reason they seem trivial is that we have written them down in the right coordinates. Let us illustrate this with an example. Consider the PDE 
\[-x^2 u_t^2+2 x^2 u_t u_x-x^2 u_x^2+2xuu_t-2xuu_x-xu_{tt}+x u_{tx}-u^2+u_t=0.\]
Even though the equation looks complicated, its symmetries are easily computed. In particular, we find that all vector fields of the form $\frac{f(t+x)}{x} \partial_u$ are symmetries. Thus, we may either look for the point-transformation that brings this to $f(t)\partial_u$ (and the PDE to $u_{tx}=u_x^2$ which is treated above), or we can find the quotient PDE directly. The algebra of differential invariants, with respect to the Lie algebra spanned by vector fields of the form $\frac{f(t+x)}{x} \partial_u$, is generated by 
\[I=t, \qquad J=x, \qquad H=u+x(u_x-u_t), \qquad \hat \partial_I=D_t, \qquad \hat \partial_J=D_x.\] 
The quotient PDE is given by $H_I=H^2$ which has general solution $H=\frac{1}{g(J)-I}$. It gives $u+x(u_x-u_t)=\frac{1}{g(x)-t}$, a new first-order PDE. Its solution is 
\[u(t,x)=\frac{1}{x} \left(\int \frac{d\tau}{\tau-g(x+t-\tau)}  +C(t+x)\right)\Bigg|_{\tau=t}.\]

\subsection{Solving the PDEs using finite-dimensional Lie algebras} \label{finitesym}
We note that  all the examples we considered in this and the previous section has a special property: The PDE we get by adding an additional differential constraint $G=0$ (of order 1 or 2) to $F=0$ is of infinite type. Since one would in general expect the result to be a finite type equation (\cite{Seiler}), all our examples are quite special. This can be explained by the infinite-dimensional symmetry Lie algebra. It is clear (due to our choice of coordinates) that neither $F$ nor $G$ will depend on $u_{tt}$. And since they are compatible, the prolonged equations will not depend on $u_{t^i}$ for every integer $i \geq 2$. 

This is different from what we got when we found the quotient of Burgers' equation in Section \ref{Burgers}. In that case we got a finite type equation with a three-dimensional solution space. We argued that any solution to the quotient PDE of Burgers' equation would determine a five-dimensional submanifold in $J^2(\mathbb R^2)$ on which the Cartan distribution was two-dimensional and completely integrable, with the symmetries acting transitively on the set of integral manifolds. Thus, given a solution of the quotient, the Lie-Bianchi theorem would let us find solutions in quadratures. What prohibited us from going through with this was our inability to solve the quotient PDE. 

In this section we will see that if we consider finite-dimensional Lie subalgebras of the infinite-dimensional ones, we get a situation similar to that of Burgers' equation, but now with partially uncoupled quotient PDEs that we can solve, since one of the two first-order PDEs is the same as the quotient with respect to the infinite-dimensional symmetry Lie algebra. We consider two examples. 

\textbf{The Hunter-Saxton equation:} 
Consider again the Hunter-Saxton equation $(u_t+uu_x)_x=u_x^2/2$, but now with the three-dimensional symmetry Lie algebra spanned by  $\partial_x$, $t \partial_x+\partial_u$ and $t^2 \partial_x+2t \partial_u$. It is a Lie subalgebra of the one we already considered in Section \ref{HS}. 
In addition to the invariants $I=t,J=u_x,H=u_{xx}$ we found in Section \ref{HS}, we now have an additional second-order invariant $K=u_{tt}-u^2 u_{xx}+u_t u_x$. With these generators, we get two first-order syzygies instead of one: 
\[ 2 \hat \partial_I (H)-J^2 \hat \partial_J (H) + 4 JH = 0, \qquad \hat \partial_J(K) = 0\] Now we are in a similar situation as we were with  Burgers' equation, but with equations that are decoupled.
The general solution is 
\[ 16 g\left(\frac{2J}{2-IJ}\right) H-(2-IJ)^4=0, \qquad K=C(I), \]
where we use similar notation as in Section \ref{HS}. This gives two second-order differential constraints:
\[ 16 g\left(\frac{2u_x}{2-tu_x}\right) u_{xx}-(2-tu_x)^4=0, \qquad u_{tt}-u^2 u_{xx}+u_t u_x=C(t) \] Together with the Hunter-Saxton equation $(u_t+uu_x)_x=u_x^2/2$, they determine a five-dimensional submanifold of $J^2(\mathbb R^2)$. The restriction of the Cartan distribution to this manifold is a two-dimensional integrable distribution.

\textbf{Liouville's equation:} 
Consider now $u_{tx}+e^u=0$ with its three-dimensional symmetry Lie algebra $\langle \partial_t, t \partial_t-\partial_u,t^2 \partial_t-2t \partial_u\rangle$. We use the differential invariants $I=x,J=u_x,H=u_{xx}$ as in Section \ref{sym2}. In addition we have one more second-order invariant $K=(2u_{tt}-u_t^2)e^{-2u}$. The quotient PDE is given by the first-order system
\[ \hat \partial_J(H)-J=0, \qquad \hat \partial_I(K)+H \hat \partial_J(K)+2JK=0.\]
The first equation gives $H=J^2/2+g(I)$. Inserting this into the second equation gives us a pure first-order PDE on $K$. 

Again, any solution to this system will give two additional differential constraints that together with $u_{tx}+e^u=0$ determine a five-dimensional manifold in $J^2(\mathbb R^2)$ on which the restriction of the Cartan distribution is two-dimensional. We have Lie algebra acting transversally on the distribution, as in the case for the Hunter-Saxton equation, but now the Lie algebra is not solvable, and we can not use the Lie-Bianchi theorem. This type of situation is treated in \cite{LychaginIntegrability}.

\section{Conclusion} \label{Conclusion}
We have shown, using several examples, how the theory of differential invariants can give us insight into the solution space of PDEs, and in some cases even lets us write down the general solution. The idea is very general and natural, and we summarize it here in an informal way.

Given a PDE $\mathcal E$ and a Lie algebra $\mathfrak g$ of symmetries, the obvious thing to do is to look for objects defined on $\mathcal E$ that are $\mathfrak g$-invariant. The scalar differential invariants are among the simplest invariant objects, and they can be generated by a finite set, as a differential algebra. In general, the algebra is not freely generated, so there are differential syzygies giving relations between the generators and their derivatives. The differential syzygies define a PDE, called the quotient PDE, whose solutions  correspond to equivalence classes of solutions of $\mathcal E$. 

Each solution of the quotient gives us additional differential constraints that can be added to the defining equation for $\mathcal E$. This amounts to restricting to an orbit of the $\mathfrak g$-action on the solution space of $\mathcal E$, resulting in a new PDE on which $\mathfrak g$ acts transitively on solutions. 
In cases where the Lie algebra of symmetries gives a quotient that is nontrivial and significantly different from $\mathcal E$, we can use the quotient as a stepping-stone to get insight into the space of solutions of $\mathcal E$. 

 We illustrated, with several examples, how the ideas put forward in this paper can be implemented in practice. For the special class of PDEs that we considered, we saw that solution of second-order PDEs could be obtained by solving two first-order PDEs, a situation very similar to that of symmetry reduction of ODEs. It is worth noting that even when we are not able to find general solutions using these ideas, they will give us a better understanding of the PDE under consideration. 

\vspace{0.5cm}

\noindent \textbf{Acknowledgements:} This project was supported by the Czech Science Foundation (GA\v{C}R no. 19-14466Y).

\section*{Appendix: Quotient of Burgers' equation}
Let us compute the differential invariants and quotient PDE for Burgers' equation with respect to its five dimensional Lie algebra of symmetries. The following four functions are invariant:
\begin{gather*}
I =\frac{ u_{xxx}^3}{u_{xx}^4}, \qquad J = \frac{u_{xxx} u_{xxxx}}{u_{xx}^3}, \qquad 
H = \frac{u_{xxxxx}}{u_{xx}^2}, \qquad K= \frac{u_{xxx}^2 u_{xxxxxx}}{u_{xx}^5}
\end{gather*}
Note that the expressions by which we have written down the invariants hide the fact that $I$ and $J$ are second-order invariants while $H$ and $K$ are of third order. 

If we take $I$, $J$ and $H$ as generators of the algebra of differential invariants, together with the invariant derivations $\hat \partial_I$ and $\hat \partial_J$, the quotient can be written as a second-order differential syzygy:
\begin{gather*}
I^2 \big(4 I-3 J\big)^2 \hat \partial_I^2(H)+2 I \big(4 I-3 J\big) \big((3 J-H) I-J^2\big) \hat \partial_J \hat \partial_I (H) \\+\big((3 J-H) I-J^2\big)^2 \hat \partial_J^2(H)+I \big((9-2 I) I+6 (I-J)^2\big) \hat \partial_I (H) \\-I \big((2 (I-J)) H-10 I+J (2 J-3)\big) \hat \partial_J(H)+2 (H-5) I^2-15 I J=0
\end{gather*}
Then the last third-order invariant $K$ is given by 
\[-\left(I (4 I-3 J) \hat \partial_I (H)+((3 J-H) I-J^2) \hat \partial_J(H)-2 I H\right).\]



\begin{thebibliography}{11}
\footnotesize

\bibitem{BlumanODE}
G. Bluman, {\it A Reduction Algorithm for an Ordinary Differential Equation Admitting a Solvable Lie Group}, SIAM J. Appl. Math. {\bf 50}, 1689-1705 (1990). 

\bibitem{HS}
J.K. Hunter, R. Saxton, {\it Dynamics of Director Fields}, SIAM Journal on Applied Mathematics {\bf 51}, 1498-1521 (1991).

\bibitem{KV}
I.S. Krasil'shchik, A.M. Vinogradov (eds.), {\it Symmetries and Conservation Laws for Differential Equations of Mathematical Physics}, AMS (1999). 

\bibitem{KL2}
B. Kruglikov, V. Lychagin, {\it Global Lie-Tresse theorem}, Selecta Mathematica {\bf 22}, 1357-1411 (2016).

\bibitem{KL1}
B. Kruglikov, V. Lychagin, {\it Geometry of Differential equations\/}, Handbook of Global Analysis, Ed. D.Krupka, D.Saunders, Elsevier, 725-772 (2008).

\bibitem{RedBook}
A. Kushner, V. Lychagin, V. Rubtsov, {\it Contact Geometry and Non-linear Differential Equations}, Cambridge University Press (2007).

\bibitem{LychaginIntegrability}
V. Lychagin, {\it Symmetries and Integrals}, to appear in the proceedings of the Summer School Wisla 19. 

\bibitem{HSsymmetries}
M. Nadjafikhah, F. Ahangari, {\it Symmetry Analysis and Conservation Laws for the Hunter-Saxton Equation}, Commun. Theor. Phys. {\bf 59}, 335-348 (2013). 

\bibitem{MorozovHS}
O.I. Morozov, {\it Contact Equivalence of the Generalized Hunter-Saxton Equation and the Euler-Poisson Equation}, \textit{arXiv:math-ph/0406016}. 

\bibitem{O}
P. Olver, {\it Equivalence, Invariants, and Symmetry}, Cambridge University Press (1995).


\bibitem{Pavlov}
M.V. Pavlov, {\it The Calogero equation and Liouville-type equations}, Theoretical and Mathematical Physics {\bf 128}, 927-932 (2001).

\bibitem{Seiler}
W.M. Seiler, {\it Involution and Symmetry Reductions}, Mathl. Comput. Modelling {\bf 25}, 63-73 (1997). 

\bibitem{Evolutionary}
S.I. Svinolupov, V.V. Sokolov, {\it Factorization of evolution equations}, Russ. Math. Surv. {\bf 47}, 127-162 (1992). 



\end{thebibliography}
\end{document}